\documentclass[11pt]{amsart}
\usepackage{amsmath,amsthm}
\usepackage{graphics}
\usepackage{latexsym}
\usepackage{verbatim}
\usepackage{amsmath}
\usepackage{amsthm}
\usepackage{amssymb}
\usepackage{epsfig}
\usepackage{enumerate}
\usepackage{epstopdf}
\usepackage{hyperref}
\usepackage[]{hyperref}
\hypersetup{urlcolor=blue, citecolor=red}
\usepackage{dsfont}
\usepackage[table]{xcolor}
\usepackage{multirow}
\usepackage{float}
\usepackage{tikz}
\usetikzlibrary{arrows, calc, shapes, positioning, decorations.pathreplacing}
\usepackage{comment}
\usepackage{subfigure}

\newcommand\ds{\displaystyle}
\newcommand{\ihalf}[0]{{i+1/2}}
\newcommand{\imhalf}[0]{{i-1/2}}
\newcommand{\D}[0]{{\rm{d}\!}}
\newcommand{\R}{\mathbb{R}}
\newcommand{\N}[0]{\mathbb{N}}
\newcommand{\qqand}[0]{ \qquad{\text{and}} \qquad}

\renewcommand{\d}[0]{{\rm{d}}}

\ifpdf
  \DeclareGraphicsExtensions{.eps,.pdf,.png,.jpg}
\else
  \DeclareGraphicsExtensions{.eps}
\fi
\newtheorem{theorem}{Theorem}[section]
\newtheorem{corollary}[theorem]{Corollary}
\newtheorem{lemma}[theorem]{Lemma}
\newtheorem{proposition}{Proposition}

\newtheorem{definition}[theorem]{Definition}
\newtheorem{rem}{Remark}
\newcommand{\mA}{\mathcal{A}}
\newcommand{\mC}{\mathcal{C}}
\newcommand{\mD}{\mathcal{D}}
\newcommand{\mE}{\mathcal{E}}
\newcommand{\mF}{\mathcal{F}}
\newcommand{\mG}{\mathcal{G}}
\newcommand{\mR}{\mathcal{R}}

\newcommand{\curlyF}{\mathcal{F}}
\newcommand{\curlyG}{\mathcal{G}}
\newcommand{\curlyX}{\mathcal{X}}
\newcommand{\frakS}{\mathfrak{S}}

\title[Convergence of a Finite Volume Scheme]{Convergence of a Finite Volume Scheme for a System of Interacting Species with Cross-Diffusion}

\author[Jos\'e A. Carrillo, Francis Filbet, and Markus Schmidtchen]{}

\subjclass[2010]{Primary: 74S10; 65M12; 92C15; Secondary: 45K05; 92D25; 47N60}
\keywords{Finite volume methods; Integro-partial differential
  equations; Population dynamics (general); Developmental biology, pattern formation;}

\email{carrillo@maths.ox.ac.uk}
\email{francis.filbet@math.univ-toulouse.fr}
\email{schmidtchen@ljll.math.upmc.fr}

\begin{document}
\maketitle

\centerline{\scshape Jos\'e A. Carrillo}
\medskip
{\footnotesize
    \centerline{Mathematical Institute, University of Oxford}
    \centerline{Oxford OX2 6GG, United Kingdom}
} 

\medskip

\centerline{\scshape Francis Filbet}
\medskip
{\footnotesize
    \centerline{Institut de Math\'ematiques de Toulouse, Universit\'e Paul Sabatier}
    \centerline{Toulouse, France}
}

\medskip

\centerline{\scshape Markus Schmidtchen}
\medskip
{\footnotesize
    \centerline{Laboratoire Jacques-Louis Lions, Sorbonne Universit\'e}
    \centerline{ 4 place Jussieu, 75005 Paris, France}
}

\begin{abstract}
	In this work we present the convergence of a positivity preserving semi-discrete finite volume scheme for a coupled system of two non-local partial differential equations with cross-diffusion. The key to proving the convergence result is to establish positivity in order to obtain a discrete energy estimate to obtain  compactness. We numerically observe the convergence to reference solutions with a first order accuracy in space. Moreover we recover segregated stationary states in spite of the regularising effect of the self-diffusion. However, if the self-diffusion or the cross-diffusion is strong enough, mixing occurs while both densities remain continuous.
\end{abstract}


\section{Introduction}
\label{sec:1}

In this paper we develop and analyse a numerical scheme for  the following non-local interaction system with cross-diffusion and self-diffusion
\begin{align}
\label{eq:crossdiffsystem}
	\left\{
	\begin{array}{r}
	\ds\frac{\partial \rho}{\partial t} = \ds\frac{\partial}{\partial x}\left(\rho \frac{\partial}{\partial x}
        \left(W_{11}\star\rho + W_{12}\star \eta + \nu(\rho+ \eta)\right)\,+\,
        \dfrac\epsilon2 \frac{\partial \rho^2}{\partial x} \right),
\\[1.75em]
	\ds\frac{\partial \eta}{\partial t}= \ds\frac{\partial}{\partial x}\left(\eta \frac{\partial}{\partial x}
       \left( W_{22}\star\eta + W_{21}\star \rho + \nu(\rho+ \eta)
                                            \right)  \,+\,
        \dfrac\epsilon2 \frac{\partial\eta^2}{\partial x} \right),
\end{array}\right.
\end{align}	
governing the evolution of two species $\rho$ and $\eta$ on an
interval $(a,b) \subset \R$ for $t\in[0,T)$. The system is equipped
with nonnegative initial data $\rho^0,\eta^0 \in L_+^1(a,b) \cap
L_+^\infty(a,b)$. We denote by $m_1$ the mass of $\rho_0$ and by $m_2$ the mass of $\eta_0$, respectively,
$$
m_1\,=\, \int_a^b \rho_0(x)\,\d x, \quad \text{and} \quad
m_2 \,=\, \int_a^b \eta_0(x) \,\d x.
$$

On the boundary $x=a$ and $b$, we prescribe no-flux boundary conditions
\begin{align*}
\left\{
\begin{array}{r}
\ds\rho\,\frac{\partial}{\partial x} \left(W_{11}\star\rho + W_{12}\star \eta +
  \nu(\rho+ \eta)+ \epsilon \rho\right) = 0,
\\
\,\\
\ds\eta\, \frac{\partial}{\partial x} \left(W_{22}\star\eta + W_{21}\star \rho + \nu(\rho+ \eta) + \epsilon \eta\right) = 0,
\end{array}\right.
\end{align*}
such that the total mass of each species is conserved with respect to
time $t\geq 0$. While the self-interaction potentials $W_{11}, W_{22}\in C_b^2(\R)$ 
model the interactions among individuals of the same species (also
referred to as \emph{intraspecific} interactions), the
cross-interaction potentials $W_{12},W_{21}\in C_b^2(\R)$ encode the
interactions between individuals belonging to  different species,
\emph{i.e.} \emph{interspecific} interactions. Here $C_b^2(\R)$ denotes the set of twice continuously differentiable functions on $\R$ with bounded derivatives. Notice that the convolutions $W_{ij}\star \psi$, with $\psi$ a density function defined on $[a,b]$, are defined by extending the density $\psi$ by zero outside the interval $[a,b]$. The two positive parameters $\epsilon, \nu > 0$ determine the strengths of the
self-diffusion and the cross-diffusion of both species, respectively. Nonlinear diffusion, be it self-diffusion or cross-diffusion, is biologically relevant. As a matter of fact, around the second half of the $20^{\mathrm{th}}$ century biologists found that the dispersal rate of certain insects depends on the density itself, leading to the nonlinear diffusion terms we incorporated in the model, cf. \cite{Mor50, Mor54, Ito52, Car71, GN75}. At the same time we would like to stress that the self-diffusion terms are relevant for the convergence analysis below.

It is the
interplay between the non-local interactions of both species and their
individual and joint size-exclusion, modelled by the non-linear
diffusion \cite{BGHP85, BGH87, BGH87a, HPVolume, CCVolume, BBRW17},
that leads to a large variety of behaviours including complete phase
separation or mixing of both densities in both stationary
configurations and travelling pulses \cite{BDFS, CHS17}.

While their single species counterparts have been studied quite
intensively \cite{MEK99, CMcCV03, TBL06, CCH14} and references therein,  related two-species models like the system of our interest, Eq. \eqref{eq:crossdiffsystem}, have only recently gained considerable attention \cite{DFF13, BDFS, CHS17, DFEF17, CFSS17}. One of the most striking phenomena of these interaction models with cross-diffusion is the possibility of phase separation. Since the seminal papers \cite{GP84, BGHP85} established segregation effects for the first time for the purely diffusive system corresponding to \eqref{eq:crossdiffsystem} for $W_{ij}\equiv 0$, $i,j\in\{1,2\}$ and $\epsilon=0$, many generalisations were presented. This includes reaction-(cross-)diffusion systems \cite{BDPM10, BHIM12, CFSS17} and references therein, and by adding non-local interactions \cite{BDFS, CHS17, DFEF17, BBP17} and references therein. Ref. \cite{DFEF17} have established the existence of weak solutions to a class of non-local systems under a strong coercivity assumption on the cross-diffusion also satisfied by system \eqref{eq:crossdiffsystem}.

Typical applications of these non-local models comprise many
biological contexts such cell-cell adhesion  \cite{PBSG, MT15,CCS17}, for
instance, as well as tumour models \cite{GerCha, DTGC},
but also the  formation of the characteristic stripe patterns of
zebrafish can be modelled by these non-local models \cite{VS}. Systems of this kind are truly ubiquitous in nature and we remark that `species' may not only refer to biological species but also to a much wider class of (possibly inanimate) agents such as planets, physical or chemical particles, just to name a few.  

Since system \eqref{eq:crossdiffsystem} is in conservative form a finite volume scheme is a natural choice as a numerical method. This is owing to the fact that, by construction, finite volume schemes are locally conservative: due  to the divergence theorem, the change in density on a test cell has to equal the sum of the in-flux and the out-flux of the same cell. There is a huge literature on finite volume schemes, first and foremost  \cite{EGH00}. Therein, the authors give a detailed description of the construction of such methods and address convergence issues. Schemes similar to the one proposed in Section \ref{sec:numerical_scheme} have been studied in \cite{BCF12} in the case of nonlinear degenerate diffusion equations in any dimension. A similar scheme for a  system of two coupled PDEs was proposed in \cite{CHF07}. Later, the authors in \cite{CCH15} generalised the scheme proposed in \cite{BCF12} including both local and non-local drifts. The scheme was then extended to two species in \cite{CHS17}. All the aforementioned schemes have in common that they preserve nonnegativity -- a property that is also crucial for our analysis.

Before we define the finite volume scheme we shall present a formal energy estimate for the continuous system. The main difficulty in this paper is to establish positivity  and reproducing the continuous energy estimate at the discrete level.
The remainder of the introduction is dedicated to presenting the aforementioned energy estimate. Let us consider
\begin{align*}
\ds	\frac{\d}{\d t} \int_a^b \rho \log \rho \,\d x &= \ds\int_a^b
                                                         \log \rho
                                                         \,\frac{\partial
                                                         \rho}{\partial t}  \,\d x\\
	&=\ds\int_a^b \log \rho \,\frac{\partial}{\partial x}
            \left(\rho \frac{\partial}{\partial x}\left( W_{11}\star\rho + W_{12}\star
            \eta + \nu(\rho+ \eta)+ \epsilon \rho\right)\right) \d x\\
	&= \ds- \int_a^b \rho\, \frac{\partial}{\partial x}\left(W_{11}\star\rho + W_{12}\star \eta + \nu(\rho+ \eta)+ \epsilon \rho\right)\, \frac{\partial}{\partial x}(\log \rho) \,\d x,
\end{align*}
 where the second equality holds due to the no-flux boundary conditions.
Upon rearranging we get
\begin{align*}
	\frac{\d }{\d t} \int_a^b \rho\log\rho \,\d x \,+\, \nu \, \int_a^b \frac{\partial}{\partial x} (\rho +\eta)\frac{\partial\rho}{\partial x}\,\d x  \,+\, \epsilon \int_a^b \left|\frac{\partial\rho}{\partial x}\right|^2\,\d x \,=\, -\int_a^b (W_{11}'\star\rho + W_{12}'\star \eta)  \, \frac{\partial\rho}{\partial x}\,\d x.
\end{align*}
A similar computation for $\eta$ yields
\begin{align*}
	\frac{\d }{\d t} \int_a^b \eta\log\eta \,\d x \,+\, \nu \, \int_a^b \frac{\partial}{\partial x} (\rho +\eta)\, \frac{\partial\eta}{\partial x}\,\d x  \,+\, \epsilon \, \int_a^b \left|\frac{\partial\eta}{\partial x} \right|^2\,\d x \,=\, -\int_a^b (W_{22}'\star\eta + W_{21}'\star \rho) \, \frac{\partial\eta}{\partial x}  \,\d x,
\end{align*}
whence, upon adding both, we obtain
$$
\begin{array}{rcl}
\ds	\frac{\d }{\d t} \int_a^b  \left[\rho\log\rho+\eta\log\eta\right] \,\d x
  \,+\,\ds \nu\int_a^b\left|\frac{\partial \sigma}{\partial x} \right|^2 \d
  x \,+\,
  \epsilon\,\int_a^b \left( \left|\frac{\partial\rho}{\partial x} \right|^2 \,+\, \left|\frac{\partial\eta}{\partial x} \right|^2 \right)\,\d
  x \,=\,\ds \mathcal{D}_\rho + \mathcal{D}_\eta,
\end{array}
$$
where $\sigma = \rho + \eta$ and
$$
\left\{
	\begin{array}{l}
		\mathcal{D}_\rho := -\ds\int_a^b (W_{11}'\star\rho + W_{12}'\star
  \eta) \,\frac{\partial\rho}{\partial x}  \, \d x, \\[1em]
  		\mathcal{D}_{\eta} := -\ds \int_a^b (W_{22}'\star\eta + W_{21}'\star \rho) \, \frac{\partial\eta}{\partial x} \,  \d x,
	\end{array}
\right.
$$
denote the advective parts associated to $\rho$ and $\eta$, respectively. The advective parts can be controlled by using the weighted Young inequality to get
\begin{align*}
\ds	|\mathcal{D}_\rho|&= \ds\left|\int_a^b (W_{11}'\star\rho + W_{12}'\star \eta) \frac{\partial\rho}{\partial x}  \, \d x\right|\\[1em]
&\leq \ds\frac{1}{2\alpha}\int_a^b {|W_{11}'\star\rho + W_{12}'\star \eta|^2} \, \d x\,+\,\frac{\alpha}{2}\int_a^b\left| \frac{\partial\rho}{\partial x}  \right|^2 \, \d x,		
\end{align*}
for some $\alpha>0$. In choosing $0<\alpha<\epsilon$ we obtain
\begin{align}
	\label{eq:cts_inequality}	
	\frac{\d }{\d t} \int_a^b  \left[\rho\log\rho+\eta\log\eta\right] \,\d x
  \,+\, \nu \int_a^b \left|\frac{\partial \sigma}{\partial x} \right|^2\,\d x \,+\, \left(\epsilon-\frac\alpha2\right)\,\int_a^b \left(\left|\frac{\partial\rho}{\partial x}\right|^2\,+\,\left|\frac{\partial\eta}{\partial x}\right|^2\right)\,\d x \,\leq\, \frac{C_\rho + C_\eta}{2\alpha},
\end{align}
where $C_\rho = \|W_{11}'\star\rho + W_{12}'\star \eta\|_{L^2}^2$ and $C_\eta = \|W_{22}'\star\eta + W_{21}'\star \rho\|_{L^2}^2$.
From the last line, Eq. \eqref{eq:cts_inequality}, we may deduce bounds on the gradient of each species as well as on their sum. As mentioned above the crucial ingredient for this estimate is the positivity of solutions.\\

The rest of this paper is organised as follows. In the subsequent section we present a semi-discrete finite volume approximation of system \eqref{eq:crossdiffsystem} and we present the main result, Theorem \ref{thm:convergence}. Section \ref{sec:apriori_estimates} is dedicated to establishing positivity and to the derivation of a priori estimates. In Section \ref{sec:convergence} we obtain compactness, pass to the limit, and identify the limiting functions as weak solutions to system \eqref{eq:crossdiffsystem}. We conclude the paper with a numerical exploration in Section \ref{sec:numerical_results}. We study the numerical order of accuracy and discuss stationary states and phase segregation phenomena.

\section{Numerical scheme and main result}
\label{sec:numerical_scheme}
In this section we introduce the semi-discrete finite volume scheme for system \eqref{eq:crossdiffsystem}. To begin with, let us introduce our notion of weak solutions.
\begin{definition}[Weak solutions.] 
\label{def:WeakSolution}
A couple of functions $(\rho,\eta)$ $\in $ $L^2(0,T;H^1(a,b))^2$ is a
weak solution to system \eqref{eq:crossdiffsystem} if it satisfies
\begin{subequations}
	\begin{align}
	\begin{split}
		-\int_a^b \rho_0 \,\varphi(0,\cdot) \,\d x \,=\! \int_0^T
                \int_a^b \left[\rho\, \left(\frac{\partial \varphi}{\partial t}  \,+\, \left(-\nu \frac{\partial \sigma}{\partial x}
		\,+\,\frac{\partial V_1}{\partial
                  x} \right)\,\frac{\partial \varphi}{\partial x} \right) \,+\,\frac{\epsilon}{2}\, \rho^2\, \frac{\partial^2 \varphi}{\partial x^2} \right]\,\d x \,\d t,
	\end{split}
	\end{align}
	and
	\begin{align}
	\begin{split}
		-\int_a^b \eta_0 \,\varphi(0,\cdot) \, \d x \,=\,\! \int_0^T
                \int_a^b \left[\eta \left(\frac{\partial \varphi}{\partial t}
                  + \left(-\nu \frac{\partial \sigma}{\partial x}  \,+\,
                    \frac{\partial V_2}{\partial x} \right)\,
                  \frac{\partial \varphi}{\partial x} \right) \,+\, \frac{\epsilon}{2}\,\eta^2 \, \frac{\partial^2 \varphi}{\partial x^2} \right]\,\d x \d t,
	\end{split}
	\end{align}
\end{subequations}
respectively, for any $\varphi\in C_c^\infty([0,T)\times(a,b); \R)$. Here we have set $V_k = - W_{k\, 1} \star \rho - W_{k \, 2} \star \eta$, for $k\in\{1,2\}$, and $\sigma = \rho + \eta$, as above.
\end{definition}

Notice that the existence of weak solutions to system
\eqref{eq:crossdiffsystem} will follow directly from the convergence
of the numerical solution. Indeed, our analysis relies on a compactness
argument which does not suppose a priori existence of solution to system \eqref{eq:crossdiffsystem}.

To this end we first define the following space discretisation of the domain.
\begin{definition}[\label{def:spacediscretisation}Space discretisation]
	To discretise space, we introduce the mesh 
	\begin{align*}
		\mathcal{T}:	= \bigcup_{i\in I} C_i,
	\end{align*}
	where the control volumes are given by $C_i=[x_\imhalf, x_\ihalf)$ for all $i \in I:=\{1,\ldots, N\}$. We assume that the measure of the control volumes are given by $|C_i| = \Delta x_i = x_{i+1/2} - x_{i-1/2}>0$, for all $i\in I$. Note that $x_{1/2}=a$, and $x_{N+1/2}=b$.
\end{definition}
\begin{figure}[ht!]
	\centering
	\begin{tikzpicture}
		\draw[-] (1,0) -- (9,0);
		\foreach \x in {1,...,4}
     		\draw[-, thick] (2*\x,-0.2) -- (2*\x,0.2);

     	\draw (2,-0.1) node[below] {$x_{i-3/2}$};
     	\draw (4,-0.1) node[below] {$x_{i-1/2}$};
     	\draw (6,-0.1) node[below] {$x_{i+1/2}$};
     	\draw (8,-0.1) node[below] {$x_{i+3/2}$};

		\foreach \x in {1,...,3}
     		\draw[-] (2*\x+1,-0.1) -- (2*\x+1,0.1);

     	\draw (3,-0.05) node[below] {$x_{i-1}$};
     	\draw (5,-0.05) node[below] {$x_{i}$};
     	\draw (7,-0.05) node[below] {$x_{i+1}$};

		\draw[decorate,decoration={brace}] ($(4+0.05,0.4)$) -- node[above]{$C_i$} ++(1.9,0);
		
		\draw[decorate,decoration={brace}] ($(2+0.05,0.4)$) -- node[above]{$C_{i-1}$} ++(1.9,0);
		
		\draw[decorate,decoration={brace}] ($(6.05,0.4)$) -- node[above]{$C_{i+1}$} ++(1.9,0);
		
		\draw[-] (-0.5,0) -- (0.5,0);
		\draw[-, dotted, thick] (0.5,0) -- (1,0);
		\draw [fill=black] (0,0) circle [radius=0.05];
		\draw (0,-0.05) node[below] {$x_{1/2} = a$};

		\draw[-, dotted, thick] (9,0) -- (9.5,0);
		\draw[->] (9.5,0) -- (10.5,0);
		\draw [fill=black] (10,0) circle [radius=0.05];
		\draw (10,-0.05) node[below] {$x_{N+1/2} = b$};
	\end{tikzpicture}
	\caption{Space discretisation according to Definition \ref{def:spacediscretisation}.}
\end{figure}
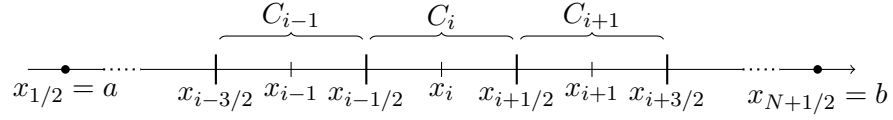

\noindent
We also define $x_i = (x_{i+1/2} + x_{i-1/2})/2$ the centre of cell $C_i$ and set $\Delta
x_{i+1/2}=x_{i+1}-x_i$ for  $i=1,\ldots,N-1$. We assume that the mesh is
regular in the sense that there exists $\xi\in (0,1)$
such that for $h := \max_{1\leq i \leq N} \{\Delta x_i\}$
\begin{equation}
\label{admissible}
	\xi\,h \,\leq \,\Delta x_{i} \,\leq\, h,
\end{equation}
and, as a consequence, $\xi\,h \,\leq \,\Delta x_{\ihalf} \,\leq\, h$, as well.

On this mesh we shall now define the semi-discrete finite volume approximation of system \eqref{eq:crossdiffsystem}.  The discretised initial data are given by the cell averages of the continuous initial data, \emph{i.e.}
\begin{align}
\label{eq:initial_data}
	\rho_i^0 := \frac{1}{\Delta x_i}\int_{C_i} \rho_0(x)\,\d x, \qquad \text{and}\qquad 	\eta_i^0 := \frac{1}{\Delta x_i} \int_{C_i} \eta_0(x)\,\d x,
\end{align}
for all $i\in I$.  Throughout, we write $\rho_i$ (resp. $\eta_i$) to denote the approximations of the two densities on the $i$-th finite volume cell, $C_i$. Next, we introduce the discrete versions of the cross-diffusion and the interaction terms. We set
\begin{align}
	\label{eq:nonlocal_part}
	\left\{
	\begin{array}{ll}
	\displaystyle
	(V_1)_i:=-\sum_{j=1}^N \Delta x_j \,\big( W_{11}^{i-j}\rho_j +  W_{12}^{i-j} \eta_j \big),\\[1.5em]
	\displaystyle	
	(V_2)_i:=-\sum_{j=1}^N  \Delta x_j \,\big( W_{22}^{i-j}\eta_j +  W_{21}^{i-j} \rho_j \big),
	\end{array}
	\right.
\end{align}
where 
\begin{align}
	\label{eq:discrete_interaction_potential}
	W_{kl}^{i-j} \,=\,\frac{1}{\Delta x_j}\, \int_{C_j} W_{kl}(|x_i-s|)\d s,
\end{align}
for $k,l=1,2$, and 
\begin{align}
	\label{eq:defU}
	U_i := -\big(\rho_i + \eta_i),
\end{align}
for the cross-diffusion term, respectively. Then the scheme reads
\begin{subequations}
\label{eq:scheme}
\begin{align}
	\label{eq:scheme_evol}
	\left\{
	\begin{array}{l}
	\displaystyle
	\frac{\d \rho_i}{\d t}(t)=\displaystyle- \frac{\mF_{\ihalf}(t) - \mF_{i-1/2}(t)}{\Delta x_i},\\[1em]
	\displaystyle
	\frac{\d \eta_i}{\d t}(t)=\displaystyle - \frac{\mG_{\ihalf}(t) - \mG_{i-1/2}(t)}{\Delta x_i},
	\end{array}
	\right.
\end{align}
for $i\in I$. Here the numerical fluxes are given by
\begin{align}
	\label{eq:numerical_fluxes}
	\left\{
	\begin{array}{l}
	\displaystyle
	\mF_{\ihalf}(t)=	\displaystyle \left[\nu\,(\d U)_{\ihalf}^+ +
                       (\d V_1)_{\ihalf}^+\right] \,\rho_i \,+\,
                       \left[\nu\,(\d U)_{\ihalf}^- + (\d
                       V_1)_{\ihalf}^-\right]\,\rho_{i+1}\\ [1.5em]
	\phantom{\mF_{\ihalf}=	}- \ds\frac{\epsilon}{2} \,\frac{\rho_{i+1}^2 - \rho_i^2}{\Delta x_{i+1/2}},\\[1.5em]
	\displaystyle
	\mG_{\ihalf}(t) =	\displaystyle \left[ \nu\,(\d U)_{\ihalf}^+
                       +(\d V_2)_{\ihalf}^+\right]\,   \eta_i \,+\,
                       \left[\nu\, (\d U)_{\ihalf}^- + (\d V_2)_{\ihalf}^-\right]\,\eta_{i+1}\\[1.5em]
	\phantom{\mG_{\ihalf} =	}- \ds\,\frac{\epsilon}{2}\, \frac{\eta_{i+1}^2 - \eta_i^2}{\Delta x_{i+1/2}},
	\end{array}
	\right.
\end{align}
for $i = 1,\ldots, N-1$, with the numerical no-flux boundary condition
\begin{align}
	\label{eq:num_noflux}
	\mF_{1/2}(t) = \mF_{N+1/2}(t) &= 0, \quad \text{and} \quad \mG_{1/2}(t) = \mG_{N+1/2}(t) = 0,	
\end{align}
\end{subequations}
where we introduced the  discrete gradient $\d u_\ihalf$ as
$$
\d u_\ihalf\,:=\,\frac{u_{i+1}-u_i}{\Delta x_{i+1/2}}.
$$ 
As usual, we use $(z)^\pm$ to denote the positive (resp. negative) part of $z$, \textit{i.e.}
\begin{align*}
	(z)^+ := \mathrm{max}(z,0), \qquad \text{and} \qquad (z)^-:=\mathrm{min}(z,0).
\end{align*}

At this stage, the numerical flux \eqref{eq:numerical_fluxes} may look
strange since 
\begin{itemize}
\item the cross-diffusion term is approximated as a convective term
  using that 
$$
\frac{\partial}{\partial x}\left( \rho \, \frac{\partial}{\partial
    x}\left( \rho + \eta \right) \right) \,=\,
\frac{\partial}{\partial x}\left( \rho \, \frac{\partial \sigma}{\partial
    x} \right)
$$
with $\sigma=\rho+\eta$ and $\frac{\partial \sigma}{\partial x}$ is considered
as a velocity field. This treatment has already been used in
\cite{BCF12} and allows to preserve the positivity of both discrete
densities $(\rho,\eta)$ (see Lemma \ref{lem:nonnegative}), which is crucial for the convergence
analysis. 
\item In this new formulation, the velocity field is split in two
  parts both treated by an upwind scheme. One part comes from the
  cross-diffusion part, and the second one comes from the non-local
  interaction fields. This splitting is crucial to recovering a consistent
  dissipative term for the discrete energy estimate corresponding to Eq.  \eqref{eq:cts_inequality}.
\end{itemize}

\begin{definition}[\label{def:ApproximateSolutions}Piecewise constant approximation]
	For a given mesh $\mathcal{T}_{h}$ we define the approximate solution to system \eqref{eq:crossdiffsystem} by
	\begin{align*}
		\rho_{h}(t,x):= \rho_i(t),\qquad \text{and} \qquad \eta_{h}(t,x):=\eta_i(t),
	\end{align*}
	for all $(t,x) \in [0,T]\times C_i$, with $i=1,\ldots,N$. Moreover, we define the following approximations of the gradients
	\begin{align*}
		\d \rho_{h}(t,x) = \frac{\rho_{i+1}-\rho_i}{\Delta x_{i+1/2}}, \qquad \mbox{and}\qquad \d \eta_{h}(t,x) = \frac{\eta_{i+1}-\eta_i}{\Delta x_{i+1/2}}
	\end{align*}
	for $(t,x)\in[0,T)\times [x_i,x_{i+1})$, for
        $i=1,\ldots,N-1$. Furthermore, in order to define $\d \rho_h$
        and $\d \eta_h$ on the whole interval $(a,b)$ we set them to zero
         on $(a,x_1)$ and  $(x_N,b$). 
\end{definition}
Notice that the discrete gradients $(\d \rho_h,\d \eta_h)$ are
piecewise constant just like $(\rho_h,\eta_h)$ however not on the same partition
of the interval  $(a,b)$.  In a similar fashion we define the piecewise constant interpolation of the discrete advection fields, \emph{i.e.},
\begin{align*}
	\d V_{k,\, h}(x)  \,=\, (\d V_{k})_\ihalf, 
\end{align*}
for all $x \in [x_i, x_{i+1})$, for $i=1,\ldots, N-1$, and zero at the boundary.

We have set out all definitions necessary to formulate the convergence of the numerical scheme \eqref{eq:scheme}.

\begin{theorem}[\label{thm:convergence}Convergence to a weak solution.]
	Let $\rho_0, \eta_0 \in L_+^1(a,b) \cap L_+^\infty(a,b)$ be
        some initial data  and $Q_T := (0,T)\times (a,b)$. Then, 
\begin{itemize}
\item [$(i)$] there exists  a nonnegative approximate solution
        $(\rho_{h}, \eta_{h})$ in the sense of Definition
        \ref{def:ApproximateSolutions};

\item[$(ii)$]  up to a subsequence, this
        approximate solution converges strongly in $L^2(Q_T)$ to
        $(\rho,\eta) \in L^2(Q_T)$, where $(\rho,\eta)$ is  a weak solution as in
        Definition \ref{def:WeakSolution}. Furthermore we have $\rho$, $\eta
        \in L^2(0,T; H^1(a,b))$;
\item[$(iii)$] as a consequence system \eqref{eq:crossdiffsystem} has a weak solution.
\end{itemize}
\end{theorem}

\section{\label{sec:apriori_estimates}A priori estimates}
This section is dedicated to deriving a priori estimates for our system. In order to do so we require the positivity of  approximate solutions and their conservation of mass, respectively. The following lemma guarantees these properties.
\begin{lemma}[Existence of nonnegative solutions and conservation of mass]
\label{lem:nonnegative}
\label{lem:conservation_of_m}
Assume that the initial data $(\rho_0, \eta_0)$ are non-negative. Then
there  exists a unique  nonnegative approximate
        solution $(\rho_{h}, \eta_{h})_{h>0}$  to the  scheme (\ref{eq:scheme_evol})-(\ref{eq:num_noflux}). Furthermore, the finite volume scheme conserves the initial mass of both densities.
\end{lemma}
\begin{proof}
	On the one hand we notice that the right-hand side of  (\ref{eq:scheme_evol})-(\ref{eq:numerical_fluxes})
	is  locally Lipschitz with respect to $(\rho_i,\eta_i)_{1\leq
          i\leq N}$. Hence, we may apply the Cauchy-Lipschitz theorem to obtain a unique continuously differentiable local-in-time solution.
 
	On the other hand to prove that this solution is global in
        time, we show the nonnegativity of the solution together with
        the conservation of mass and argue by contradiction. 

On a given mesh, let some initial data, $\rho_i(0), \eta_i(0)\geq 0$,
be given for $i=1, \ldots N$. We rewrite the scheme in the following way.
	\begin{align}
	\label{eq:PosSchemeRewritten}
		\frac{\d \rho_i}{\d t}(t) &= -\frac{\mF_\ihalf - \mF_\imhalf}{\Delta x_i} \,=\,\frac{1}{\Delta x_i}\big(A_i \,\rho_i \,+\, B_i \,\rho_{i+1} \,+\, C_i\, \rho_{i-1}\big),
	\end{align}
	where 
	\begin{align*}
		\left\{
		\begin{array}{l}
			A_i=\displaystyle \nu \,(\d U)_\imhalf^- \,+\,
                              (\d V_1)_\imhalf^-\,-\,\nu\,(\d
                              U)_\ihalf^+ \,-\, (\d V_1)_\ihalf^+ \,-\, 
                              \,\frac\epsilon  2 \left(\frac{\rho_i}{\Delta x_{i+1/2}} + \frac{\rho_i}{\Delta x_{i-1/2}}\right),
\\[1em]
		B_i =\displaystyle -\,\nu\,(\d U)_\ihalf^- \,-\, (\d
                      V_1)_\ihalf^- \,+\, \epsilon\,
                      \frac{\rho_{i+1}}{2\Delta x_{i+1/2}},
\\[1em]
		C_i =\displaystyle \nu \,(\d U)_\imhalf^+ \,+\, (\d V_1)_\imhalf^+ \,+\, \epsilon \frac{\rho_{i-1}}{2\Delta x_{i-1/2}}.
		\end{array}
		\right.
	\end{align*}

	Then let $t^\star\geq 0$ be the maximal time for all densities
        to remain nonnegative, \emph{i.e.}
$$
	t^\star = \sup\left\{ t \geq 0 \, |\, \rho_i(s) \geq 0,\,
          \mbox{for all } s \in [0,t], \mbox{ and } i=1,\ldots,N
\right\}.
$$

If $t^\star < \infty$, then there exists a nonincreasing sequence $(t_k)_{k\in\N}$
such that $t_k>t^\star$,  $t_k\rightarrow t^\star$ as $k\rightarrow \infty$ and there
exists $i_k\in \{1,\ldots, N\}$ verifying 
$$
\rho_{i_k}(t_k)<0,\quad \forall k\in\N.
$$
Since the index $i_k$ takes a finite number of  integer values, we can
extract a nonincreasing subsequence  of $(t_k)_{k\in\N}$
still labeled in the same manner such that there exists an index  $j_0\in
\{1,\ldots,N\}$ and 
$$
\rho_{j_0}(t_k)<0,\quad \forall k\in\N,
$$
where $t_k\rightarrow t^\star$, as $k$ goes to infinity.

Also note by continuity of $(\rho_i)_{1\leq i\leq N}$, we have that
$\rho_{i}(t^\star)\geq 0$ for any $i\in\{1,\ldots, N\}$. 

By the above computation, Eq. \eqref{eq:PosSchemeRewritten}, we see
that, if  $\rho_{j_0+1}(t^\star)>0$ or respectively $\rho_{j_0-1}(t^\star)>0$, then
either $B_{j_0}(t^\star)>0$ or  respectively $ C_{j_0}(t^\star)>0$ and
\begin{align*}
\frac{\d \rho_{j_0}}{\d t}(t^\star) & =\,\frac{1}{\Delta
                                      x_{j_0}}\left(A_{j_0}
                                      \,\rho_{j_0}(t^\star) \,+\,
                                      B_{j_0}\, \rho_{j_0+1}(t^\star)
                                      \,+\, C_{j_0}
                                      \,\rho_{j_0-1}(t^\star)\right)
\\
		                            &=\,\frac{1}{\Delta x_{j_0}}\left(B_{j_0} \,\rho_{j_0+1}(t^\star) \,+\, C_{j_0}\, \rho_{j_0-1}(t^\star)\right)
		\,>\, 0,
\end{align*}
hence there exists $\tau>0$ such that for any $t\in
[t^\star,t^\star+\tau)$, we have $\rho_{j_0}(t) > \rho_{j_0}(t^\star)=0$,
which cannot occur since $\rho_{j_0}$ is continuous and for $t_k>t^\star$,
$\rho_{j_0}(t_k) < 0$ for any $k\in\N$ with $t_k\rightarrow t^\star$
when $k$ goes to infinity.

If $\rho_{j_0-1}(t^\star)=\rho_{j_0}(t^\star)=\rho_{j_0+1}(t^\star)=0$
then by uniqueness of the solution, we have that $\rho_{j_0}\equiv 0$
for $t\geq t^\star$, which contradicts again that $\rho_{j_0}(t_k) < 0$ for any $k\in\N$ large enough.

Finally we get  the conservation of mass, 
\begin{align*}
		\ds\frac{\d}{\d t}\int_a^b \rho_{h}(t,x)\d x 
		&= \ds\sum_{i=1}^N\Delta x_i \frac{\d}{\d t}\rho_i\\
		&= \ds\sum_{i = 1}^N \Delta x_i \frac{\mF_{\ihalf} - \mF_{\imhalf}}{\Delta x_i} \,=\, \mF_{N+1/2}-\mF_{1/2} \,=\, 0,
\end{align*}
	by the no-flux condition. Analogously, the second species
        remains nonnegative and its mass is conserved as well. As a consequence of the control of the  $L^1$-norm of $(\rho_h,\eta_h)$ we can extend the local solution to a global, nonnegative solution.
\end{proof}

\noindent
Now, we are ready to study the evolution of the energy of the system on the semi-discrete level. The remaining part of this section is dedicated to proving the following lemma -- an estimate similar to \eqref{eq:cts_inequality} for the semi-discrete scheme \eqref{eq:scheme}. 

\begin{lemma}[Energy control\label{lem:entropy_control}]
Consider a solution of the semi-discrete scheme
\eqref{eq:scheme_evol}-\eqref{eq:numerical_fluxes}. Then we have
\begin{align*}
	\ds \frac{\d }{\d t} \sum_{i=1}^N\Delta x_i [\rho_i \log \rho_i + \eta_i \log\eta_i]\ds+\,\sum_{i=1}^{N-1}\!\Delta x_{i+1/2}\left[\nu\, |\d U_\ihalf|^2 
		\,+\, \frac{\epsilon}{4}\left(  |\d \rho_\ihalf|^2
		\,+\,  |\d \eta_\ihalf|^2\right)
				\right]\,\leq\, C_{\epsilon},
\end{align*}
where the constant $C_{\epsilon}>0$ is given by 
\begin{equation}
\label{def:c}
C_{\epsilon} \,=\, \frac{(b-a)}{\epsilon} \left(\,\left( \|W_{11}'\|_{L^\infty}+\|W_{21}'\|_{L^\infty} \right)^2\, m_1^2 \,+\,
\left(  \|W_{12}'\|_{L^\infty} +  \|W_{22}'\|_{L^\infty}\right)^2\, m_2^2\,\right).
\end{equation}
\end{lemma}
\begin{proof}
Upon using the scheme, Eq. \eqref{eq:scheme_evol}, we get
\begin{align*}
	 \ds\frac{\d}{\d t} \sum_{i=1}^N \Delta x_i\,\rho_i \log \rho_i =\ds-\sum_{i=1}^N(\mF_{\ihalf} - \mF_{i-1/2}) \log\rho_i,
\end{align*}
due to the conservation of mass, ensured by Eq. \eqref{eq:num_noflux}.
By discrete integration by parts and the no-flux condition, Eq. \eqref{eq:num_noflux}, we obtain
\begin{align*}
\ds\frac{\d}{\d t} \sum_{i=1}^N \Delta x_i\,\rho_i \log \rho_i
                &=\ds\sum_{i=1}^{N-1}  \Delta x_{i+1/2}\,  \mF_{\ihalf}\, \D\log
             \rho_{\ihalf}
\\
	&=\ds\nu\sum_{i=1}^{N-1} \Delta x_{i+1/2} \bigg(
		(\d U)_{\ihalf}^+ \rho_i \,+\,(\d
            U)_{\ihalf}^- \rho_{i+1}  \bigg)	   \rm{d}\! \log\rho_{\ihalf}
\\
&\quad+\ds\sum_{i=1}^{N-1} \Delta x_{i+1/2} \bigg( (\d V_1)_{\ihalf}^+ \rho_i + (\d
            V_1)_{\ihalf}^- \rho_{i+1}  \bigg)	   \rm{d}\! \log\rho_{\ihalf},
\\
&\quad-\ds\frac{\epsilon}{2}\,\sum_{i=1}^{N-1} \left(\rho_{i+1}^2 - \rho_i^2\right)\, 	   \rm{d}\! \log\rho_{\ihalf},
\end{align*}
where, in the last equality, we substituted the definition of the numerical flux, Eq. \eqref{eq:numerical_fluxes}.
Let us define
\begin{align}
\label{eq:tilde_rho_ihalf}
	\tilde \rho_\ihalf := \left\{
		\begin{array}{ll}
			\dfrac{\rho_{i+1} - \rho_i}{\log\rho_{i+1} - \log \rho_i}, & \text{if $\rho_i \neq \rho_{i+1}$},\\[1em]
			\dfrac{\rho_i + \rho_{i+1}}{2}, & \text{else},
		\end{array}
		\right.
\end{align}
for $i\in\{1,\ldots,N-1\}$, and note that then $\tilde \rho_\ihalf \in
[\rho_i, \rho_{i+1}]$ by concavity of the $\log$.  Here, and throughout, we use the shorthand notation $[x,y]:=[\min(x,y), \max(x,y)]$. Reordering the terms, we obtain
\begin{align}
\label{eq:someeqn}
\begin{split}
	\frac{\d}{\d t} \sum_{i=1}^N &\Delta x_i \rho_i \log \rho_i \,-\, \sum_{i=1}^{N-1} \Delta x_{i+1/2}\, \left[ \nu \,\d U_\ihalf \tilde \rho_\ihalf \,-\,\frac{\epsilon}{2}\,\d \rho_\ihalf^2 \right]\D\log\rho_\ihalf\\
	=&\,\nu\sum_{i=1}^{N-1}\Delta x_{i+1/2}\left((\d U)_{\ihalf}^+ (\rho_i - \tilde \rho_\ihalf) 
		  \,+\, (\d U)_{\ihalf}^-
                  (\rho_{i+1}-\tilde\rho_\ihalf)\right)\, \D\log\rho_{\ihalf}\\
       &+\,\sum_{i=1}^{N-1}\Delta x_{i+1/2}\left( (\d V_1)_{\ihalf}^+ (\rho_i - \tilde \rho_\ihalf) 
		  \,+\,(\d V_1)_{\ihalf}^- (\rho_{i+1}-\tilde\rho_\ihalf) \right)\,\D\log\rho_{\ihalf}\\
	&+\, \sum_{i=1}^{N-1}\Delta x_{i+1/2} \, \tilde\rho_\ihalf\,\d V_{1,\ihalf} \,\D\log\rho_{\ihalf}.
\end{split}
\end{align}
Thus, using $\tilde \rho_\ihalf \in
[\rho_i, \rho_{i+1}]$ and the monotonicity of $\log$, we note that
\begin{align}
	\label{eq:numerical_artifacts}
	\left\{
	\begin{array}{r}
	\displaystyle
	(\rho_i - \tilde \rho_\ihalf) \D\log \rho_\ihalf \big(\nu(\d U)_{\ihalf}^+ + (\d V_1)_{\ihalf}^+\big) \leq 0,\\[1em]
	\displaystyle
	(\rho_{i+1} - \tilde \rho_\ihalf) \D\log \rho_\ihalf \big(\nu (\d U)_{\ihalf}^- + (\d V_1)_{\ihalf}^-\big) \leq 0.
	\end{array}
	\right.
\end{align}
This is easy to see, for, if $\rho_i = \rho_{i+1}$, we observe $ \D\log \rho_\ihalf=0$ and Eqs. \eqref{eq:numerical_artifacts} hold with equality. In the case of $\rho_i< \rho_{i+1}$ we observe
\begin{align*}
	\underbrace{(\rho_i - \tilde\rho_\ihalf)}_{\leq  0} \underbrace{\D\log\rho_\ihalf}_{\geq 0} \underbrace{\big(\nu(\d U)_\ihalf^+ + (\d V_1)_\ihalf^+ \big)}_{\geq 0} \leq 0,
\end{align*}
while, for $\rho_i > \rho_{i+1}$ there also holds
\begin{align*}
	\underbrace{(\rho_i - \tilde\rho_\ihalf)}_{\geq  0} \underbrace{\D\log\rho_\ihalf}_{\leq 0} \underbrace{\big(\nu(\d U)_\ihalf^+ + (\d V_1)_\ihalf^+ \big)}_{\geq 0} \leq 0,
\end{align*}
whence we infer the inequality. The same argument can be applied in order to obtain the second line of Eq. \eqref{eq:numerical_artifacts}. Thus we may infer from Eq. \eqref{eq:someeqn} that
\begin{align*}
\begin{split}
	\frac{\d}{\d t} \sum_{i=1}^N &\Delta x_i \rho_i \log \rho_i -  \sum_{i=1}^{N-1} \Delta x_{i+1/2} \left(\nu \tilde \rho_\ihalf \, \d U_\ihalf \,-\,\frac{\epsilon}{2}\,\d \rho_\ihalf^2\right)\, \D\log\rho_\ihalf\\
	&\leq\sum_{i=1}^{N-1}\Delta x_{i+1/2} \,\tilde \rho_\ihalf\,\D\log\rho_{\ihalf} \,\d V_{1,\ihalf}.
\end{split}
\end{align*}
Note that the definition of $\tilde \rho_\ihalf$ in Eq. \eqref{eq:tilde_rho_ihalf}, is consistent with the case $\rho_i = \rho_{i+1}$ and there holds
\begin{align*}
	\tilde \rho_\ihalf \,\D\log\rho_\ihalf \,=\, \d \rho_\ihalf,
\end{align*}
whence we get 
\begin{align*}
\begin{split}
	\frac{\d}{\d t} \sum_{i=1}^N &\Delta x_i \,\rho_i \,\log \rho_i
        \,-\, \nu\,\sum_{i=1}^{N-1} \Delta x_{i+1/2} \,\d\rho_\ihalf\,\d
        U_\ihalf  \\
& +\,\frac{\epsilon}{2}\, \sum_{i=1}^{N-1} \Delta x_{i+1/2} \,\d \rho_\ihalf^2 \,\D\log\rho_\ihalf	\,\leq\,\sum_{i=1}^{N-1}\Delta x_{i+1/2} \,\d\rho_{\ihalf} \,\d V_{1,\ihalf}.
\end{split}
\end{align*}
Furthermore, we notice that 
$$
	\frac{1}{2}\,\d \rho_\ihalf^2 \D\log\rho_\ihalf  \,=\;  \frac{\rho_{i+1} +
          \rho_i}{2\,\tilde \rho_\ihalf} |\d\rho_{i+1/2}|^2\,\geq\,\frac{1}{2}
        |\d\rho_{i+1/2}|^2,
$$
where we employed Eq. \eqref{eq:tilde_rho_ihalf}. Hence we have 
\begin{align}
\label{eq:291116_0954}
\begin{split}
	\frac{\d}{\d t} \sum_{i=1}^N &\Delta x_i \,\rho_i \,\log \rho_i
        \,-\, \nu\,\sum_{i=1}^{N-1} \Delta x_{i+1/2} \,\d\rho_\ihalf\,\d
        U_\ihalf  \\
& +\,\frac{\epsilon}{2}\, \sum_{i=1}^{N-1} \Delta x_{i+1/2}\, |\d \rho_\ihalf|^2 	\,\leq\,\sum_{i=1}^{N-1}\Delta x_{i+1/2} \,\d\rho_{\ihalf} \,\d V_{1,\ihalf}.
\end{split}
\end{align}
A similar computation can be  applied to the second species, which  yields
\begin{align}
\label{eq:220916_1137}
\begin{split}
	\frac{\d}{\d t} \sum_{i=1}^N &\Delta x_i \,\eta_i \,\log \eta_i
        \,-\,  \nu \sum_{i=1}^{N-1} \Delta x_{i+1/2}\, \d \eta_\ihalf \,\d
        U_\ihalf \\
& +\,\frac{\epsilon}{2}\sum_{i=1}^{N-1}\Delta x_{i+1/2} \,|\d \eta_\ihalf|^2	\,\leq\,\sum_{i=1}^{N-1}\Delta x_{i+1/2} \,\d \eta_\ihalf \d V_{2,\ihalf}.
\end{split}
\end{align}
\noindent
Upon adding up equations \eqref{eq:291116_0954} and \eqref{eq:220916_1137}, we obtain
\begin{align}
\label{eq:220916_1212}
\begin{split}
	\frac{\d }{\d t}\sum_{i=1}^N &\Delta x_i\,[\rho_i \log \rho_i
        + \eta_i \log \eta_i] \,-\, \nu\,\sum_{i=1}^{N-1}\Delta
        x_{i+1/2}\,\d U_\ihalf \,\d\left(\rho+\eta\right)_\ihalf \\
	&\ds\,+\,\frac{\epsilon}{2}\,\sum_{i=1}^{N-1}\Delta x_{i+1/2}
        \,\bigg(|\d \rho_\ihalf|^2  \,+\, |\d
        \eta_\ihalf|^2 \bigg)\,\leq\, \mR_h,
\end{split}
\end{align}	
where $\mR_h$ is given by
$$
\mR_h \,=\, \sum_{i=1}^{N-1}\Delta x_{i+1/2} \,\left[ \d V_{1,\ihalf}
\,\d\rho_\ihalf \,+\, \d V_{2,\ihalf} \,\d\eta_\ihalf\right].
$$
Finally we notice  that
\begin{equation}
	\label{eq:19_11_17_0911}
	\mR_h\,\leq\,\frac{1}{2}\sum_{i=1}^{N-1} \Delta x_{i+1/2} \left[\frac{|\d V_{1,\ihalf}|^2}{\alpha} \,+\, \alpha\,|\d \rho_\ihalf|^2 \,+\, \frac{|\d V_{2,\ihalf}|^2}{\alpha} \,+\, \alpha \,|\d \eta_\ihalf|^2\right],
\end{equation} 
for any $\alpha>0$, by Young's inequality. Observing, that for $k=1,\,2$,
\begin{align*}
|\d V_{k,\ihalf}| &\leq \left|\sum_{j=1}^{N-1} \rho_j\,\int_{C_j}
  \frac{W_{k1}(x_{i+1}-y)-W_{k1}(x_{i}-y)}{\Delta x_{i+1/2}} \d
  y\right| 
\\
&\quad + \left|\sum_{j=1}^{N-1} \eta_j\,\int_{C_j}
  \frac{W_{k2}(x_{i+1}-y)-W_{k2}(x_{i}-y)}{\Delta x_{i+1/2}} \d
  y\right| 
\\
 &\leq \|W_{k1}'\|_{L^\infty} \, m_1 \,+\, \|W_{k2}'\|_{L^\infty} \, m_2
\end{align*}
and by conservation of positivity and mass, it gives for $k=1,\,2$,
\begin{eqnarray*}
\frac{1}{2\alpha}\sum_{i=1}^{N-1}\Delta x_{i+1/2} |\d V_{k,\ihalf}|^2 \,\leq\, \frac{ (b-a)}{2\,\alpha} \, \left(m_1\,\|W_{k1}'\|_{L^\infty} \, +\,  m_2\|W_{k2}'\|_{L^\infty}  \right)^2.
\end{eqnarray*}
Thus Eq. \eqref{eq:19_11_17_0911} becomes
$$
\mR_h \,\leq\, \frac{\alpha}{2}\sum_{i=1}^{N-1} \Delta x_{\ihalf}\left(|\d
  \rho_\ihalf|^2+|\d \eta_\ihalf|^2\right)  \,+ \, C_{2\alpha},
$$
where $C_{2\alpha}$ is given in Eq. \eqref{def:c}. Finally, substituting the latter
estimate into  Eq. \eqref{eq:220916_1212},  we obtain, upon using Eq. \eqref{eq:defU},
\begin{align*}
	 \frac{\d }{\d t} \sum_{i=1}^N\Delta x_i [\rho_i \log \rho_i + \eta_i \log\eta_i]
		\,&+\, \nu\,\sum_{i=1}^{N-1}\!\Delta x_{\ihalf} |\d  U_\ihalf|^2 \\
&+ \frac{\epsilon \,-\,\alpha}{2}\, \sum_{i=1}^{N-1}\!\Delta x_{\ihalf} \left( |\d \rho_\ihalf|^2\,+\, |\d \eta_\ihalf|^2\right)
				\,\leq\, C_{2\alpha},
\end{align*}
for any solution $(\rho_i)_{i\in I}, (\eta_i)_{i\in I}$ of the
semi-discrete scheme \eqref{eq:scheme}. Hence choosing
$\alpha=\epsilon/2$ concludes the proof. 
\end{proof}
\begin{corollary}[A priori bounds]
\label{cor:apriori}
	Let $(\rho_i)_{i\in I}, (\eta_i)_{i\in I}$ be solutions of the
        semi-discrete scheme \eqref{eq:scheme}. Then there exists a
        constant $C>0$ such that 
	\begin{align*}
		\int_0^T \sum_{i=1}^{N-1} \Delta x_{\ihalf} \left(\,\frac\epsilon4|\d\rho_\ihalf|^2 + \frac\epsilon4|\d\eta_\ihalf|^2 + \,\nu|\d U_\ihalf|^2 \right) \,\d
          t&\leq C.
	\end{align*}
\end{corollary}
\begin{proof}
Using the fact that  $x \log x \geq -\log(e)/e$, \emph{i.e.} $x\log x$
is bounded from below, yields
\begin{align*}
\sum_{i\in I}	\Delta x_i \, [\rho_i \log \rho_i + \eta_i \log\eta_i]
  (t) \geq -2\frac{\log(e)}{e}\,  (b-a)=: -C_1.
\end{align*}
Hence the discrete version of the classical entropy functional is bounded from below. Therefore, we integrate
the inequality of Lemma \ref{lem:entropy_control} in time and get 
\begin{align*}
\int_0^T \sum_{i=1}^{N-1} &\Delta x_{i+1/2}\left[\nu\, |\d U_\ihalf(t)|^2 
		\,+\, \frac{\epsilon}{4}\left(  |\d \rho_\ihalf(t)|^2
		\,+\,  |\d \eta_\ihalf(t)|^2\right)
				\right]\,\d t \\
&\leq C_\epsilon\, T \,+\,C_1\,+\,\int_a^b\rho_0 |\log \rho_0| + \eta_0 |\log\eta_0|\d x,
\end{align*}
which proves the statement.
\end{proof}

Thanks to a classical discrete Poincar\'e inequality \cite[Lemma 3.7]{EGH00} and \cite{BCF}, we get
uniform $L^2$-estimates on the discrete approximation $(\rho_h,\eta_h)_{h>0}$.
\begin{lemma}
\label{lem:aprioriL2}
	Let $(\rho_i)_{i\in I}, (\eta_i)_{i\in I}$ be the numerical
        solutions obtained from scheme \eqref{eq:scheme}. Then there holds
	\begin{align*}
		\|\rho_{h}\|_{L^2(Q_T)} \,+\, \|\eta_{h}\|_{L^2(Q_T)} \leq C,
	\end{align*}
	for some constant $C>0$ independent of $h>0$.
\end{lemma}

\section{\label{sec:convergence}Proof of Theorem \ref{thm:convergence}}
This section is dedicated to proving compactness of both species, the fluxes, and the regularising porous-medium type diffusion. Upon establishing the compactness result we identify the limits as weak solutions in the sense of Definition \ref{def:WeakSolution}.  

First by application of Lemma \ref{lem:nonnegative}, we get existence
and uniqueness  of a nonnegative  approximate solution
$(\rho_h,\eta_h)$ to
(\ref{eq:scheme_evol})-(\ref{eq:numerical_fluxes}). Hence the first
item of Theorem \ref{thm:convergence} is proven. Now let us
investigate the asymptotic $h\rightarrow 0$.

\subsection{Strong compactness of approximate solutions}
We shall now make use of the above estimates in order to obtain strong
compactness of both species, $(\rho_{h},\eta_{h})$ in $L^2(Q_T)$. 

\begin{lemma}[Strong compactness in $L^2(Q_T)$]
\label{lem:strongconvergence}
	Let $(\rho_{h}, \eta_{h})_{h>0}$ be the approximation to system \eqref{eq:crossdiffsystem} obtained by the semi-discrete scheme \eqref{eq:scheme}. Then there exist functions $\rho, \eta \in L^2(Q_T)$ such that
	\begin{align*}
		\rho_{h}\rightarrow \rho,\qquad \text{and} \qquad \eta_{h} \rightarrow \eta,
	\end{align*}
	strongly in $L^2(Q_T)$, up to a subsequence.
\end{lemma}
\begin{proof}

	We invoke the compactness criterion by Aubin and Lions \cite{CJL14}.
	Accordingly, a set $P\subset L^2(0,T;B)$ is relatively compact if $P$ is bounded in $L^2(0,T;X)$ and the set of derivatives $\{\partial_t \rho \big| \rho\in P\}$ is bounded in a third space $L^1(0,T;Y)$, whenever the involved Banach spaces satisfy $X \hookrightarrow\hookrightarrow B \hookrightarrow Y$, 	\emph{i.e.} the first embedding is compact and the second one continuous. For our purpose we choose $X:=BV(a,b)$, $B:=L^2(a,b)$, and $Y:=H^{-2}(a,b)$. 
The first embedding is indeed compact,  \emph{e.g.} Ref. \cite[Theorem 10.1.4]{ABM14} and the second one is continuous. 
		
		In the second step we show the time derivatives are bounded in $L^1(0,T;H^{-2}(a,b))$. To this end, let $\varphi\in C_c^\infty((a,b))$. Throughout, we write $\langle \cdot, \cdot\rangle$ for $\langle \cdot, \cdot\rangle_{H^{-2}, H^2}$ for the dual pairing. Making use of the scheme, there holds
	\begin{align*}
	\begin{split}
		\left\langle\frac{\d  \rho_{h}}{\d t}, \varphi\right\rangle 
		&=\sum_{i=1}^N \int_{C_i} \frac{\d \rho_i }{\d t} \,\varphi \,\d x =-\sum_{i=1}^N  \frac{\mF_\ihalf - \mF_\imhalf}{\Delta x_i}  \int_{C_i}\varphi \,\d x,
	\end{split}
	\end{align*}
	having used the scheme, Eq. \eqref{eq:scheme_evol}. Next we
        set
$$
	\varphi_i :=  \frac{1}{\Delta x_i}\int_{C_i}\varphi \,\d x,
$$
perform a discrete integration by parts and use the no-flux boundary conditions, Eq. \eqref{eq:num_noflux}, to obtain
$$		
\left\langle\frac{\d \rho_{h}}{\d t} , \varphi\right\rangle 
		\,=\, \sum_{i=1}^{N-1}  \mF_\ihalf \, \left( \varphi_{i+1} -
                  \varphi_i\right).
$$
Using the definition of the numerical flux, Eq. \eqref{eq:numerical_fluxes}, we  get
	\begin{align*}
	\begin{split}
		\left\langle\frac{\d \rho_h}{\d t}, \varphi\right\rangle
                \,=\,&\sum_{i=1}^{N-1}\left[\left(\nu(\d U)_\ihalf^++(\d V_1)_\ihalf^+\right) \,\rho_i  \,+\,
                  \left(\nu(\d U)_\ihalf^- \,+\, (\d V_1)_\ihalf^-\right)\,\rho_{i+1} \right]\, \left(
                  \varphi_{i+1} -\varphi_i\right)
\\
		&- \frac{\epsilon}{2} \,\sum_{i=1}^{N-1}\frac{\rho_{i+1}^2 - \rho_i^2}{\Delta x_{i+1/2}}\, \left(
                  \varphi_{i+1} -\varphi_i\right).
	\end{split}
	\end{align*}
	Let us begin with the self-diffusion part. Using the
        Cauchy-Schwarz inequality, we estimate the discrete gradient and $\rho$
        itself by Corollary \ref{cor:apriori} and Lemma \ref{lem:aprioriL2}
	\begin{align}
		\label{eq:CS_estimate}
		\begin{split}
		\frac{\epsilon}{2}  &\int_0^T \sum_{i=1}^{N-1}\frac{\rho_{i+1}^2 - \rho_i^2}{\Delta x_{i+1/2}}\, \left(
                  \varphi_{i+1} -\varphi_i\right) \,\d t \\
		&\leq \,\frac{\epsilon}{2} \left\|\frac{\partial \varphi}{\partial x} \right\|_{L^\infty} \int_0^T \sum_{i=1}^{N-1} \Delta x_\ihalf |\d \rho_\ihalf| \, (\rho_{i+1} + \rho_i)\, \d t\\
		&\leq \,\frac{\epsilon}{2} \left\|\frac{\partial \varphi}{\partial x} \right\|_{L^\infty} \, \left(\int_0^T \sum_{i=1}^{N-1} \Delta x_\ihalf |\d \rho_\ihalf|^2 \d t\right)^{1/2} \,\left(\int_0^T  \sum_{i=1}^{N} 4\xi^{-1} \Delta x_{i} |\rho_i|^2 \d t\right)^{1/2}\\
		&\leq  C \,\|\varphi\|_{H^2(a,b)},
		\end{split}
	\end{align}
	where we used that $\varphi' \in H^1 \subset L^\infty$  and the regularity of the mesh, $\xi>0$, \emph{cf.} Eq. \eqref{admissible}.

	Next, we address the cross-diffusion and non-local interactions terms
        using the same argument. For instance for the cross-diffusive
        part, we have 
	\begin{align*}
		\nu \int_0^T &\left|\sum_{i=1}^{N-1}\left[(\d U)_\ihalf^+ \rho_i  \,+\,(\d U)_\ihalf^- \rho_{i+1}\right]\right| \, \left|
                  \varphi_{i+1} -\varphi_i\right| \,\d t \\
		&\leq \frac{2\,\nu}{\sqrt{\xi}} \,\left\|\frac{\partial \varphi}{\partial x}\right\|_{L^\infty} \left(\int_0^T \sum_{i=1}^{N-1} \Delta x_\ihalf |\d U_\ihalf|^2 \d t\right)^{1/2}   \left(\int_0^T\sum_{i=1}^{N}\Delta x_i \rho_i^2 \d t\right)^{1/2}\\
		&\leq C\, \|\varphi\|_{H^2(a,b)},
	\end{align*}
	where we used Corollary \ref{cor:apriori} and Lemma
        \ref{lem:aprioriL2} again. The non-local interaction term is
        estimated in the same way, thus there holds
	\begin{align*}
		\int_0^T \left|\left\langle \frac{\d \rho_h}{\d t}, \varphi \right\rangle\right| \d t \leq C \|\varphi\|_{H^2(a,b)}.
	\end{align*}
	By density of $C_c^\infty((a,b))$ in $H_0^{2}(a,b)$  we may infer the boundedness of $(\frac{\d\rho_h}{\d t})_{h>0}$ in $L^1(0,T; H^{-2}(a,b))$, which concludes the proof.
\end{proof}

From the latter result we can prove the convergence of the discrete advection
field $\d V_{1,h}$ and $\d V_{2,h}$ defined as in Definition \ref{def:ApproximateSolutions}.

\begin{lemma}
\label{dVh}
For any $1\leq p \leq \infty$ and $k\in\{1,2\}$, the piecewise constant approximation  $\d
V_{k,h}$ converges strongly in $L^2(0,T;L^2(a,b))$  to $-(W_{k\, 1}^\prime\star
\rho + W_{k\, 2}^\prime\star \eta)$, where $(\rho,\eta)$
corresponds to the limit obtained in Lemma \ref{lem:strongconvergence}. 
\end{lemma} 
\begin{proof}
Let $k\in\{1,2\}$. For each $i=0,\ldots,N-1$, and $x\in [x_i, x_{i+1})$ we have
\begin{align*}
\d V_{k,\, h}(x)  \,=\, (\d V_{k})_\ihalf &= -\sum_{j=1}^{N} \int_{C_j} \frac{ W_{k\,1}(x_{i+1}-y)-
  W_{k\,1}(x_i-y)}{\Delta x_\ihalf} \,\rho_j\,\d y
\\
&\quad- \sum_{j=1}^{N} \int_{C_j} \frac{W_{k\,2}(x_{i+1}-y)- W_{k\, 2}(x_i-y)}{\Delta x_\ihalf}
\,\eta_j\,\d y.
\end{align*}
We define $V_{k, h}'$ and $V_{k}'$  as 
$$
\left\{
\begin{array}{ll} 
\ds {V}_{k, h}'(x) &\ds \,:=\, - W_{k\,1}^\prime\star \rho_h \,-\,   W_{k\,2}^\prime\star \eta_h,
\\ \, \\
\ds V_{k}'(x)  & \ds \,:=\, -W_{k\,1}^\prime\star \rho \, - \,   W_{k\,2}^\prime\star \eta.
\end{array}\right.
$$
On the one hand from the strong convergence of $(\rho_h,\eta_h)$ to $(\rho,\eta)$
in $L^2(0,T;L^2(a,b))$ and the convolution product's properties, we obtain 
\begin{equation}
\|  V_{k,
  h}'- V_{k}'\|_{L^2(0,T;L^2(a,b))}\rightarrow 0, \textrm{\,when\, }h\rightarrow
0.
\label{res:1}
\end{equation} 
On the other hand, we have for any $x\in [x_i,x_{i+1})$ 
\begin{align*}
|\d V_{k,\, h}(x)-V_{k,\, h}^\prime(x)|   &\leq \sum_{j=1}^{N} \int_{C_j} \left|\frac{ W_{k\,1}(x_{i+1}-y)-
  W_{k\,1}(x_i-y)}{\Delta x_\ihalf}-W_{k\,1}^\prime(x-y)\right| \,\rho_j\,\d y,
\\
&\quad+ \sum_{j=1}^{N} \int_{C_j} \left|\frac{W_{k\,2}(x_{i+1}-y)- W_{k\, 2}(x_i-y)}{\Delta x_\ihalf}
-W_{k\,2}^\prime(x-y)\right| \,\eta_j\,\d y,
\\
&\leq  \left( \,\|W_{k\,1}^{\prime\prime}\|_{L^\infty} m_1
        \,+\,
        \|W_{k\,2}^{\prime\prime}\|_{L^\infty}\,m_2\;\right) \, h,
\end{align*}
hence there exists a constant $C>0$ such that
$$
	|\d V_{k,\, h}(x)-V_{k,\, h}^\prime(x)|^2 \,\leq \, C^p \, h^2.
$$
Integrating over $x\in [x_i, x_{i+1})$ and summing over $i\in
\{1,\ldots,N-1\}$, we get that 
\begin{equation}
\|  \d V_{k,
  h}- V_{k,h}'\|_{L^2(0,T;L^2(a,b))}\rightarrow 0, \textrm{\,when\, }h\rightarrow
0.
\label{res:2}
\end{equation} 
Notice that $(x_1,x_{N})\subset (a,b)$ where $x_1\rightarrow a$ and
$x_N\rightarrow b$ as $h\rightarrow 0$. From Eqs. (\ref{res:1}) and (\ref{res:2}) we  get that 
$\|\d V_{k,h} - V_{k}^\prime\|_{L^2(0,T;L^2(a,b))}$ goes to zero as $h$
tends to zero.
\end{proof}

\subsection{Weak compactness for the discrete gradients}
In the previous section we have established the strong $L^2$-convergence of both species, $(\rho_{h})_{h>0}$ and $(\eta_{h})_{h>0}$. However, in order to be able to pass to the limit in the cross-diffusion term $\rho_{h} (\d \rho_{h} +  \d \eta_{h})$ we need to establish weak convergence in the discrete gradients in $L^2$. This is done in the following proposition.
\begin{proposition}[\label{prop:WeakConvergenceDerivatives}Weak convergence of the derivatives]
	The discrete spatial derivatives, defined in Definition
        \ref{def:ApproximateSolutions}, satisfy   $\d \beta_h$ 
        converges weakly to $\frac{\partial \beta}{\partial x}$ in $L^2(Q_T)$ and $\beta\in L^2(0,T;H^1(a,b))$, where $\beta \in \{\rho, \eta, U\}$
\end{proposition}
\begin{proof}
	Take $\beta\in \{\rho, \eta, U\}$, hence from Lemma \ref{lem:strongconvergence}, we
        know that $\beta_h\rightarrow \beta$ strongly in
        $L^2(Q_T)$. Furthermore, from Corollary \ref{cor:apriori} we also
        deduce that $\d \beta_h$ weakly converges to some function $r \in
        L^2(Q_T)$. 

Let us show that $\beta\in L^2(0,T, H^1(a,b))$ and
        $r=\frac{\partial\beta}{\partial x}$. First, we have for any $t\in [0,T]$ and any $\varphi\in
        \mC^\infty_c((0,T)\times(a,b))$,
	\begin{align*}
		\int_{Q_T} \beta_{h}(t)\,\frac{\partial
          \varphi}{\partial x} \,\d x &\,=\, \int_0^T \sum_{i=1}^N\beta_i(t) \,\left[ \varphi(t,x_\ihalf) - \varphi(t,x_\imhalf)\right]\, \d t\\
		&\,=\,-\int_0^T \sum_{i=1}^{N-1} \Delta x_\ihalf \,\d \beta_\ihalf(t)\, \varphi(t,x_\ihalf)\, \d t,
	\end{align*}
	having used discrete integration by parts and the fact that
        $\varphi$ is  compactly supported, \emph{i.e.} $\varphi(t,
        x_{N+1/2})=\varphi(t, x_{1/2})=0$. Then, by
        Definition \ref{def:ApproximateSolutions} on the discrete
        gradient, we may consider
	\begin{align*}
	&\left|\int_0^T\sum_{i=1}^{N-1}\int_{x_i}^{x_{i+1}} \d
                        \beta_{\ihalf} \varphi(t,x)\d x \d t +
                        \int_0^T\int_a^b \beta_{h}\,
                        \frac{\partial\varphi}{\partial x} \,\d x\, \d t \right|\\
	&\quad\,\leq\, \int_0^T\sum_{i=1}^{N-1} \int_{x_i}^{x_{i+1}} \left| \d \beta_\ihalf\right|\, \left|\varphi(t,x) - \varphi(t,x_\ihalf)\right|\,\d x \,\d t\\
	&\quad\,\leq\,\left\|\frac{\partial \varphi}{\partial x}\right\|_\infty\,  \left(\int_0^T\sum_{i=1}^{N-1}\Delta x_{i+1/2} \left|\d \beta_\ihalf\right|^2\d t\right)^{1/2} \left(\int_0^T\sum_{i=1}^{N-1}\Delta x_{i+1/2}^3\d t\right)^{1/2}\\
	&\quad\,\leq\,\left\|\frac{\partial \varphi}{\partial x}\right\|_\infty\, C^{1/2} \,T^{1/2} \,\sqrt{b-a}\,h,
\end{align*}
having used the a priori bounds, \emph{cf.} Corollary \ref{cor:apriori}.
This yields the statement, when $h\rightarrow 0$, for we have 
	\begin{align}
		\label{eq:convergence_weak_derivative}
		\int_0^T\sum_{i=1}^{N-1} \int_{x_i}^{x_{i+1}} \d
          \beta_\ihalf \varphi(t,x)\,\d x\, \d t \,+\,
          \int_0^T\int_a^b \beta_{h}\, \frac{\partial\varphi}{\partial
          x}\, \d x\, \d t \,\rightarrow\, 0,
	\end{align}
	which proves that $\d \beta_h$ converges  weakly to
        $\frac{\partial \beta}{\partial x}$, as $h\rightarrow 0$ and thus  $\beta \in L^2(0,T;H^1(a,b))$.
\end{proof}

\subsection{Passing to the limit}
We have now garnered all information necessary to  prove  Theorem \ref{thm:convergence}. For brevity we shall
only show the convergence result for $\rho$, as it follows for $\eta$
similarly, using the same arguments. Let $\varphi\in
C_c^\infty([0,T)\times (a,b))$ be a test function. We introduce the
following notations: 
$$
\left\{ 
\begin{array}{lll}
\mE_h &:=&\ds \int_0^T\int_{a}^b \rho_{h} \frac{\partial
  \varphi}{\partial t} \,\d x\, \d t \,+\, \int_a^b \rho_{h}(0)\,\varphi(0)\,\d x, 
\\ \, \\ 
\mA_h &:=&\ds \int_0^T\int_{a}^b  \d V_{1,h}
          \,\rho_h\,\frac{\partial\varphi}{\partial x}\,\d x\, \d t,
\\ \, \\ 
\mC_h &:=&\ds\nu\,\int_0^T\int_{a}^b   \d U_{h} \,\rho_h\,\frac{\partial\varphi}{\partial x}\,\d x\, \d t,
\\ \, \\ 
\mD_h &:=&\ds\frac{\epsilon}{2}\int_0^T \int_{a}^b\rho_{h}^2\,\frac{\partial^2\varphi}{\partial x^2}\,\d x\, \d t.
\end{array}\right.
$$
and 
$$
\varepsilon(h) \,:=\, \mE_h \,+\, \mA_h \,+\, \mC_h \,+\, \mD_h.
$$

On the other hand, we set 
$$
\varphi_i(t) = \frac{1}{\Delta x_i} \int_{C_i} \varphi(t,x)\,\d x,
$$ 
and multiply the scheme, Eq. \eqref{eq:scheme_evol}, by the test function and integrate in time and space to get
\begin{equation}
 \mE_h \,+\, \mA_{1,h} \,+\, \mC_{1,h} \,+\, \mD_{1,h} \,=\, 0,
\label{cata:0}
\end{equation}
where 
$$
\left\{ 
\begin{array}{lll}
\mA_{1,h} &\,:=\,&\ds\sum_{i=1}^{N-1} \int_0^T \Delta x_\ihalf\,\left[(\d V_1)_{\ihalf}^+ \,\rho_i
          + (\d V_1)_{\ihalf}^- \,\rho_{i+1}\right]\,\d\varphi_{i+1/2}(t)\, \d t,
\\ \, \\ 
\mC_{1,h} &\,:=\,&\ds\nu\,\sum_{i=1}^{N-1} \int_0^T \Delta x_\ihalf\,\left[(\d U)_{\ihalf}^+ \,\rho_i
          + (\d U)_{\ihalf}^- \,\rho_{i+1}\right]\,\d\varphi_{i+1/2}(t)\, \d t,
\\ \, \\ 
\mD_{1,h} &\,:=&\,\ds-\frac{\epsilon}{2}\sum_{i=1}^{N-1} \int_0^T \left[\rho_{i+1}^2
          - \rho_{i}^2\right]\,\d\varphi_{i+1/2}(t)\, \d t.
\end{array}\right.
$$

When $h$ tends to zero and from the strong convergence of
$(\rho_h,\eta_h)_{h>0}$ to $(\rho,\eta)$ in $L^2(Q_T)$, the strong convergence of $(\d V_{k,h})_{h>0}$ to
$V_k^\prime$ in $L^2(Q_T)$  and the weak convergence of the discrete gradient
$(\d U_h)_{h>0}$ to $-\frac{\partial \sigma}{\partial x}$ in $L^2(Q_T)$, it
is easy to see that 
\begin{eqnarray*}
\varepsilon(h) &\rightarrow& \int_0^T\int_{a}^b \left\{ \rho \left[\frac{\partial
  \varphi}{\partial t}\,+\,\left(\frac{\partial
  V_1}{\partial x} - \nu\frac{\partial
  }{\partial x}\left(\rho+\eta\right)\right)\,\frac{\partial
  \varphi}{\partial x} \right] \,+\, \frac{\epsilon}{2}\,\rho^2 \, \frac{\partial^2
  \varphi}{\partial x^2} \right\}  \,\d x\, \d t\\
&\,& \,+\, \int_a^b \rho(0)\,\varphi(0)\,\d x,
\end{eqnarray*}
when $h\rightarrow 0$. Therefore it suffices to prove that
$\varepsilon(h)\rightarrow 0$, as $h$ goes to zero, which will be
achieved by proving that $\mA_h-\mA_{1,h}$, $\mC_h-\mC_{1,h}$ and
$\mD_h-\mD_{1,h}$ vanish  in the limit $h\rightarrow 0$.

\subsubsection*{The self-diffusion part $\mD_h-\mD_{1,h}$.} On the one
hand,  after a simple integration we get
	\begin{eqnarray*}
\mD_h	&=&
            \frac{\epsilon}{2}\,\sum_{i=1}^N \,\int_0^T\rho_i^2(t)
            \,\left[\frac{\partial \varphi}{\partial x}(t,x_\ihalf)
            -\frac{\partial \varphi}{\partial x}(t,x_{i-1/2})\right]
            \,\d t 
\\
&=&
            -\frac{\epsilon}{2}\,\sum_{i=1}^{N-1} \,\int_0^T\left[\rho_{i+1}^2(t)\,-\, \rho_{i}^2(t) \right] 
            \,\frac{\partial \varphi}{\partial x}(t,x_\ihalf)\,\d t. 
	\end{eqnarray*}
Hence, we have
\begin{eqnarray*}
\mD_h-\mD_{1,h} &=& -\frac{\epsilon}{2}\,\sum_{i=1}^{N-1} \,\int_0^T\left[\rho_{i+1}^2(t)\,-\, \rho_{i}^2(t) \right] 
            \,\left[\frac{\partial \varphi}{\partial x}(t,x_\ihalf)-\d
                    \varphi_\ihalf(t)\right]\,\d t 
\end{eqnarray*}
and observing that 
$$
\left|\frac{\partial \varphi}{\partial x}(t,x_\ihalf)-\d
                    \varphi_\ihalf(t)\right| \,\leq \;\left\|
                    \frac{\partial^2 \varphi}{\partial
                      x^2}\right\|_{L^\infty} h,
$$
we obtain, in conjunction with the Cauchy-Schwarz inequality and the {\it a priori} bounds established in  Corollary \ref{cor:apriori},
and Lemma \ref{lem:aprioriL2}, that
\begin{eqnarray}
\label{resu:D}
|\mD_h-\mD_{1,h}| &\leq& \frac\epsilon2 \left\|
                    \frac{\partial^2 \varphi}{\partial
                      x^2}\right\|_{L^\infty}\,\left(\sum_{i=1}^{N-1}
  \int_0^T\Delta x_{\ihalf}|\d\rho_{\ihalf}|^2 \d t\right)^{1/2}\,
\frac{2\,\|\rho_h\|_{L^2(Q_T)}}{\xi^{1/2}}\,h \\
&\leq&\, C\, h,
\nonumber
\end{eqnarray}
in the virtue of the estimate Eq. \eqref{eq:CS_estimate}.
	
\subsubsection*{The cross-diffusion part.} 
Let us now treat the cross-diffusion part. This term is more complicated
since it involves the piecewise constant functions $\rho_h$ and $\d
U_h$, which are not defined on the same mesh. Thus, on the one hand we
reformulate the discrete cross-diffusion term 
$\mC_{1,h}$ as $\mC_{1,h}=\mC_{10,h}+\mC_{11,h}$ with
$$
\mC_{10,h} \,=\,\nu\,\sum_{i=1}^{N-1} \int_0^T \Delta x_\ihalf\, (\d U)_{\ihalf}^-
          \,\left[\rho_{i+1}-\rho_{i}\right]\,\d\varphi_{i+1/2}(t)\, \d t
$$
and
$$
\mC_{11,h} \,=\,\nu\,\sum_{i=1}^{N-1} \int_0^T \Delta x_\ihalf \,\rho_i\,\d U_{\ihalf} \,\d\varphi_{i+1/2}(t)\, \d t,
$$
where a direct computation and the application of  Corollary \ref{cor:apriori}  and Lemma \ref{lem:aprioriL2} yield
\begin{eqnarray}
\label{resu:2}
|\mC_{10,h}| &\leq&\nu\, \left\|
                    \frac{\partial \varphi}{\partial
                      x}\right\|_{L^\infty} \,\|\d U_{h}\|_{L^2(Q_T)}
                    \,\|\d\rho_{h}\|_{L^2(Q_T)} \, h\\
&\leq&  C\, h.
\nonumber 
\end{eqnarray}
On the other hand, the term $\mC_h$ can be rewritten as
$$
\mC_h \,=\,\nu\,\int_0^T\sum_{i=1}^{N-1} \d U_{\ihalf}(t)\,   \int_{x_i}^{x_{i+1}} \rho_h\, \frac{\partial\varphi}{\partial x}\,\d x\, \d t.
$$
Since  
\begin{eqnarray*}
\int_{x_i}^{x_{i+1}}   \rho_h\, \frac{\partial\varphi}{\partial x}\,\d
  x & =&
\rho_i \,\left[\varphi(t,x_\ihalf) - \varphi(t,x_i)\right] \,+\,
  \rho_{i+1}\, \left[\varphi(t,x_{i+1})   - \varphi(t,x_{\ihalf})\right],
\\
&=& \left[ \rho_i - \rho_{i+1}\right] \,\left[\varphi(t,x_\ihalf) - \varphi(t,x_i)\right] \,+\,
  \rho_{i+1}\, \left[\varphi(t,x_{i+1})   - \varphi(t,x_{i})\right],
\end{eqnarray*}
the term $\mC_h$ can be decomposed as $\mC_h=\mC_{00,h}+\mC_{01,h}$
with 
$$
\mC_{00,h} \,=\, -\nu\,\int_0^T \sum_{i=1}^{N-1}  \d U_{\ihalf}\, \left[\rho_{i+1}
     -\rho_i\right]
          \,\left[\varphi(t,x_\ihalf)- \varphi(t,x_i)\right]\,\d t 
$$
and
$$
\mC_{01,h} \,=\,\nu\,\int_0^T \sum_{i=1}^{N-1}  \d
    U_{\ihalf} \,\rho_i\,\left[\varphi(t,x_{i+1})-\varphi(t,x_i)\right]\,\d t.
$$
Similarly to (\ref{resu:2}), the first term $\mC_{00,h}$  can
be estimated as
\begin{equation}
\label{resu:3}
|\mC_{00,h} |\,\leq  C\, h,
\end{equation}
whereas the second term $\mC_{01,h}$ is compared to $\mC_{11, h}$ 
$$
|\mC_{01, h}-\mC_{11, h}|\,\leq\,\nu \int_0^T \sum_{i=1}^{N-1} \Delta x_\ihalf |\d
    U_{\ihalf}| \,\rho_i\,\left|\frac{\varphi(t,x_{i+1})-\varphi(t,x_i)}{\Delta x_\ihalf}-d\varphi_{i+1/2}(t)\right|\,\d t.
$$
Using a second order Taylor expansion of $\varphi$ at $x_i$ and
$x_{i+1}$, it yields that 
$$
\left|\frac{\varphi(t,x_{i+1})-\varphi(t,x_i)}{\Delta
    x_\ihalf}-d\varphi_{i+1/2}(t)\right|\,\leq\, C\, h,
$$
hence we get  from  Corollary \ref{cor:apriori}  and Lemma \ref{lem:aprioriL2} that
\begin{equation}
\label{resu:4}
|\mC_{01, h}-\mC_{11, h}| \leq C\, h.
\end{equation}
Gathering Eqs. (\ref{resu:2}), (\ref{resu:3}), and (\ref{resu:4}), we finally
obtain that
\begin{equation}
\label{resu:C}
|\mC_{h}-\mC_{1, h}| = |\mC_{00,h} + \mC_{01,h} - \mC_{10,h}-\mC_{11,h}|\leq C\, h.
\end{equation}
	
\subsubsection*{The advective part.} The evaluation of
$\mA_{h}-\mA_{1,h}$ is along the same lines of the
cross-diffusion terms $\mC_{h}-\mC_{1,h}$ since the latter is treated
as an advective term. Hence, thanks to Lemma \ref{dVh}, we get that
\begin{equation}
\label{resu:A}
|\mA_{h}-\mA_{1, h}| \leq C\, h.
\end{equation}

Finally by definition of $\varepsilon(h)$ and using  Eq.
(\ref{cata:0}) together with Eqs. (\ref{resu:D}),  (\ref{resu:C}), and
(\ref{resu:A}),  we obtain
\begin{eqnarray*}
|\varepsilon(h)| &=& | -(\mA_{1,h} \,+\, \mC_{1,h} \,+\, \mD_{1,h}) + \mA_{h} \,+\, \mC_{h} \,+\, \mD_{h} |
\\
&\leq & | \mA_{h} -\mA_{1,h} | \,+\, | \mC_{h} -\mC_{1,h} | \,+\,  |
        \mD_{h} -\mD_{1,h} | 
\\
&\leq&  C\, h,
\end{eqnarray*}
that is, $\varepsilon(h)\rightarrow 0$, when $h\rightarrow 0$, which
proves that $(\rho,\eta)$ is a weak solution to
Eq. \eqref{eq:crossdiffsystem}. This proves the second item of Theorem
\ref{thm:convergence}.

Finally the last item concerning the existence of solutions to
\eqref{eq:crossdiffsystem} is a direct consequence of the convergence.  


\section{A Fully Discrete Implicit Scheme}
In this section we shall comment on a discrete-in-time version of the semi-discrete scheme  \eqref{eq:scheme}. To this end we replace the time derivative in Eq. \eqref{eq:scheme_evol} by simple forward differences and obtain the following implicit and fully-discrete scheme
\begin{align}
	\label{eq:discrete_scheme_evol}
	\left\{
	\begin{array}{l}
	\displaystyle
	\frac{\rho_i^{n+1}-\rho_i^{n}}{\Delta t}=\displaystyle- \frac{\mF_{\ihalf}^{n+1} - \mF_{i-1/2}^{n+1}}{\Delta x_i},\\[1em]
	\displaystyle
	\frac{\eta_i^{n+1} - \eta_i^{n}}{\Delta t}=\displaystyle - \frac{\mG_{\ihalf}^{n+1} - \mG_{i-1/2}^{n+1}}{\Delta x_i},
	\end{array}
	\right.
\end{align}
where $\Delta t >0$. System \eqref{eq:discrete_scheme_evol} gives rise
to two approximating sequences $(\rho_i^n)_{1\leq i \leq N}$ and
$(\eta_i^n)_{1\leq i \leq N}$, for $0\leq n\leq M$ where $M:= \lceil T / \Delta t\rceil$ and the discrete time instances $t^n:=n\Delta t$; cf. Theorem \ref{thm:existence_discrete}. Here the numerical fluxes are given by
\begin{align}
	\label{eq:implicit_numerical_fluxes}
	\left\{
	\begin{array}{l}
	\displaystyle
	\mF_{\ihalf}^{n+1} =	\displaystyle \left[\nu\,(\d U^{n+1})_{\ihalf}^+ +
                       (\d V_1^{n})_{\ihalf}^+\right] \,\rho_i^{n+1} \,+\,
                       \left[\nu\,(\d U^{n+1})_{\ihalf}^- + (\d
                       V_1^n)_{\ihalf}^-\right]\,\rho_{i+1}^{n+1}\\ [1.5em]
	\phantom{\curlyF_{\ihalf}=	}- \ds\frac{\epsilon}{2} \,\frac{(\rho_{i+1}^{n+1})^2 - (\rho_i^{n+1})^2}{\Delta x_{i+1/2}},\\[1.5em]
	\displaystyle
	\curlyG_{\ihalf}^{n+1} =	\displaystyle \left[ \nu\,(\d U^{n+1})_{\ihalf}^+
                       +(\d V_2^n)_{\ihalf}^+\right]\,   \eta_i^{n+1} \,+\,
                       \left[\nu\, (\d U^{n+1})_{\ihalf}^- + (\d V_2^{n})_{\ihalf}^-\right]\,\eta_{i+1}^{n+1}\\[1.5em]
	\phantom{\curlyG_{\ihalf} =	}- \ds\,\frac{\epsilon}{2}\, \frac{(\eta_{i+1}^{n+1})^2 - (\eta_i^{n+1})^2}{\Delta x_{i+1/2}},
	\end{array}
	\right. \tag{\ref{eq:discrete_scheme_evol}a}
\end{align}
for $i = 1,\ldots, N-1$, with the numerical no-flux boundary condition
\begin{align}
\label{eq:implicit_num_noflux}
	\mF_{1/2}^{n+1} = \mF_{N+1/2}^{n+1} &= 0, \quad \text{and} \quad \mG_{1/2}^{n+1} = \mG_{N+1/2}^{n+1} = 0,\tag{\ref{eq:discrete_scheme_evol}b}
\end{align}
for $n=0,\ldots,M$. Recall that
\begin{align*}
	(\d U^{n+1})_\ihalf^{\pm} = \left(\frac{(\rho_{i+1}^{n+1} + \eta_{i+1}^{n+1}) - (\rho_{i}^{n+1} + \eta_{i}^{n+1})}{\Delta x_\ihalf}\right)^\pm,
\end{align*}
and

\begin{align*}
(\d V_{k}^n)_\ihalf^\pm &= \bigg(-\sum_{j=1}^{N} \int_{C_j} \frac{ W_{k\,1}(x_{i+1}-y)-
  W_{k\,1}(x_i-y)}{\Delta x_\ihalf} \,\rho_j^n\,\d y
\\
&\qquad - \sum_{j=1}^{N} \int_{C_j} \frac{W_{k\,2}(x_{i+1}-y)- W_{k\, 2}(x_i-y)}{\Delta x_\ihalf}
\,\eta_j^n\,\d y\bigg)^\pm,
\end{align*}
for $k=1,2$.

Similarly to Definition \ref{def:ApproximateSolutions} we define the piecewise constant interpolation by
	\begin{align*}
		\rho_{h}(t,x):= \rho_i^n,\qquad \text{and} \qquad \eta_{h}(t,x):=\eta_i^n,
	\end{align*}
	for all $(t,x) \in [t^n, t^{n+1})\times C_i$, with $i=1,\ldots,N$, and $n=0,\ldots, M$. Moreover, we define the discrete approximation of the spatial gradients as 
	\begin{align*}
		\d_x \rho_{h}(t,x) = \frac{\rho_{i+1}^n-\rho_i^n}{\Delta x_{i+1/2}}, \qquad \mbox{and}\qquad \d_x \eta_{h}(t,x) = \frac{\eta_{i+1}^n - \eta_i^n}{\Delta x_{i+1/2}},
	\end{align*}
	for $(t,x)\in[t^n,t^{n+1})\times [x_i,x_{i+1})$, for
        $i=1,\ldots,N-1$ and $n=0,\ldots,M$. As above, we set the discrete gradients to zero
         on $(a,x_1)$ and  $(x_N,b$). Furthermore, we define the discrete time derivative as  
	\begin{align*}
		\d_t \rho_{h}(t,x) = \frac{\rho_{i}^{n+1}-\rho_i^n}{\Delta t}, \qquad \mbox{and}\qquad \d_t \eta_{h}(t,x) = \frac{\eta_{i}^n - \eta_i^n}{\Delta t},
	\end{align*}
	for $(t,x)\in[t^n,t^{n+1})\times C_i$, for
        $i=1,\ldots,N$ and $n=0,\ldots,M-1$.

\begin{theorem}[\label{thm:existence_discrete}Existence and uniqueness
  result]
Let $\rho_i^0, \eta_i^0$ be nonnegative initial data with mass $m_1$ and $m_2$, respectively, and 
	assume the following time step restriction condition
	\begin{align}
		\label{eq:CFL}
16\, (m_1+m_2)\,\frac{\Delta  t}{(\xi\,h)^3}\,<\, 1.
	\end{align}
Then there exists a unique nonnegative solution $(\rho_i^n, \eta_i^n)$
to scheme \eqref{eq:discrete_scheme_evol},
        \eqref{eq:implicit_numerical_fluxes} and
        \eqref{eq:implicit_num_noflux}. 
\end{theorem}
\begin{proof}
We show existence first and prove uniqueness later. Suppose we are given $(\rho_i^n)_{1\leq i\leq N}$ and $(\eta_i^n)_{1\leq i\leq N}$ from some previous iteration. In order to construct the next iteration we shall employ Brouwer's fixed point theorem. It is easy to verify that the set
\begin{align*}
	\curlyX:=\left\{ (\rho, \eta) \in \R^{2N} \,|\, \forall 1\leq i\leq N: \rho_i, \eta_i \geq 0,  \sum_{i=1}^N \Delta x_i \rho_i \leq m_1, \text{ and } \sum_{i=1}^N \Delta x_i \eta_i \leq  m_2\right\},
\end{align*}
is a convex and compact subset of $\R^N\times\R^N$. Hence, we define the
fixed point operator $\frakS:\curlyX \rightarrow \R^N \times \R^N$ by setting $(\rho^\star, \eta^\star) = \frakS(\rho, \eta)$
where $(\rho^\star,\eta^\star)$ are implicitly given as
\begin{align*}
\rho^\star_i \,=\, \rho_i^n - \frac{\Delta t}{\Delta x_i} \left(\curlyF_\ihalf^\star - \curlyF_\imhalf^\star\right),\;\,\, \eta^\star_i = \eta_i^n - \frac{\Delta t}{\Delta x_i} \left(\curlyG_\ihalf^\star - \curlyG_\imhalf^\star\right)
\end{align*}
for any $(\rho, \eta) \in \curlyX$, where $\curlyF^\star$ and $\curlyG^\star$ denote the numerical fluxes
\begin{align}
\label{eq:LIN_implicit_numerical_fluxes}
	\left\{
	\begin{array}{l}
	\displaystyle
	\curlyF_{\ihalf}^\star =	\displaystyle \left[\nu\,(\d U^{})_{\ihalf}^+ +
                       (\d V_1^n)_{\ihalf}^+\right] \,\rho_i^{\star} \,+\,
                       \left[\nu\,(\d U)_{\ihalf}^- + (\d
                       V_1^n)_{\ihalf}^-\right]\,\rho_{i+1}^{\star}
\\ [1.5em]
	\phantom{\curlyF_{\ihalf}=	}- \ds\frac{\epsilon\,
          (\rho_{i+1}^{}+ \rho_i^{})}{2} \,\frac{\rho_{i+1}^{\star} -
          \rho_i^{\star}}{\Delta x_{i+1/2}},
\\[1.5em]
	\displaystyle
	\curlyG_{\ihalf}^{\star} =	\displaystyle \left[ \nu\,(\d U^{})_{\ihalf}^+
                       +(\d V_2^n)_{\ihalf}^+\right]\,   \eta_i^{\star} \,+\,
                       \left[\nu\, (\d U^{})_{\ihalf}^- + (\d
          V_2^n)_{\ihalf}^-\right]\,\eta_{i+1}^{\star}
\\[1.5em]
	\phantom{\curlyG_{\ihalf} =	}- \ds\,\frac{\epsilon \,(\eta_{i+1} + \eta_i)}{2}\, \frac{\eta_{i+1}^{\star} - \eta_i^{\star}}{\Delta x_{i+1/2}},
	\end{array}
	\right.
\end{align}
for $i = 1,\ldots, N-1$ where  $\d U$ is computed from $(\rho,\eta)$,  with the numerical no-flux boundary condition
\begin{align}
\label{eq:LIN_implicit_num_noflux}
	\curlyF_{1/2}^{\star} = \curlyF_{N+1/2}^{\star} &= 0, \quad \text{and} \quad \mG_{1/2}^{\star} = \mG_{N+1/2}^{\star} = 0.
\end{align}
Notice that for any given $(\rho,\eta)\in \curlyX$, the viscosity
terms in front of the discrete gradients involved in the definition of
the fluxes $\curlyF_{\ihalf}^\star$ and
$\curlyG_{\ihalf}^\star$ are indeed nonnegative, hence  the couple $(\rho^\star,\eta^\star)$ is well
defined since it corresponds to the unique solution of a classical
fully implicit scheme in time with an upwind discretisation for the convective
terms and a centred approximation for diffusive terms \cite{EGT00}.  Moreover, since $\rho^n$ and $\eta^n$ are
nonnegative and using  the monotonicity of the numerical flux with respect to
$(\rho^\star,\eta^\star)$, we prove that both densities $\rho^\star$
and $\eta^\star$ are also nonnegative. Furthermore, using the nonnegativity and the no-flux conditions, we get
\begin{align*}
	\|\rho^\star\|_{L^1} = \sum_{i=1}^N\Delta x_i \rho_i^\star = m_1, \qqand \|\eta^\star\|_{L^1} = \sum_{i=1}^N\Delta x_i \eta_i^\star = m_2, 
\end{align*}
which yields that $\frakS(\curlyX)\subset \curlyX$.   Finally,
$\frakS$ is continuous as the composition of continuous
functions. Thus, we may apply Brouwer's fixed point theorem to infer
the existence of a fixed point, $(\rho^{n+1}, \eta^{n+1})$. It now remains to show uniqueness of the fixed point. 

To treat in a systematic way the boundary conditions and simplify the presentation, we define ghost values
for $(\rho,\eta)$ by setting for $\alpha\in\{\rho,\,\eta\}$, and $k\in\{1,\,2\}$,
$$
\alpha_{N+1}\,=\,\alpha_N \quad  \alpha_{0}\,=\,\alpha_1 \quad{\rm and}\quad
(\d V_k^n)_{N+1/2} \,=\,(\d V_k^n)_{1/2} = 0.
$$
Then we consider two solutions $(\tilde \rho,
\tilde \eta)$ and $(\rho, \eta)$ to  (\ref{eq:discrete_scheme_evol}).
Setting $h(x):=x^2/2$,  $s := \rho-\tilde\rho$ and $r :=
\eta - \tilde \eta$, we  get after substituting the two solutions to
(\ref{eq:discrete_scheme_evol}), for $i=1,\ldots, N$,
\begin{align*}
	s_i = &- \frac{\Delta t}{\Delta x_i} \left((\d V_1^n)_\ihalf^+
                \,s_i + (\d V^n_1)_\ihalf^- \, s_{i+1}-(\d
                V_1^n)_\imhalf^+ \, s_{i-1} - (\d V^n_1)_\imhalf^- \, s_i\right)\\
	&+ \frac{\Delta t}{\Delta x_i} \left(\frac{[h(\rho_{i+1})
          -h(\tilde \rho_{i+1})] \,-\,  
          [h(\rho_i) -
          h(\tilde \rho_i)]}{\Delta x_\ihalf } \,-\, \frac{[h(\rho_{i})- h(\tilde
          \rho_i)] \,-\, [h(\rho_{i-1})-h(\tilde \rho_{i-1})]}{\Delta x_\imhalf}\right)\\
	&- \frac{\Delta t}{\Delta x_i} \left( (\d U)_\ihalf^+ \, s_i
          + (\d U)_\ihalf^- \, s_{i+1} -(\d U)_\imhalf^+ \, s_{i-1}
          -(\d U)_\imhalf^- \, s_{i} \right) \\
	&- \frac{\Delta t}{\Delta x_i} \left([(\d U)_\ihalf^+  - (\d
          \tilde U)_\ihalf^+]\,\tilde \rho_i \,+\, [(\d U)_\ihalf^-  - (\d
          \tilde U)_\ihalf^-]\, \tilde \rho_{i+1} \right) \\
	&+\frac{\Delta t}{\Delta x_i} \left( [(\d U)_\imhalf^+  - (\d
          \tilde U)_\imhalf^+] \,\tilde \rho_{i-1} \,+\, [(\d U)_\imhalf^- - (\d \tilde U)_\imhalf^-] \,\tilde \rho_{i} \right)
\end{align*}
and a similar relation for $(r_i)_{1\leq i\leq N}$. Applying a Taylor
expansion on $h(\rho)=h(\tilde\rho) + h'(\hat\rho)\,s$,
with $\hat\rho$ a convex combination of $\rho$ and $\tilde\rho$, we
may write  
\begin{equation}
\label{tmp:0}
\left\{
\begin{array}{ll}
(\Delta x_i + \Delta t \,A_i )\,s_i = &  \ds\Delta t \left( B_{i-1}\,s_{i-1} \,+\, C_{i+1} \, s_{i+1}  \right) \\[0.75em]
	&-\ds\Delta t \,\left([(\d U)_\ihalf^+  - (\d
          \tilde U)_\ihalf^+]\,\tilde \rho_i \,+\, [(\d U)_\ihalf^-  - (\d
          \tilde U)_\ihalf^-]\, \tilde \rho_{i+1} \right) \\[0.75em]
	&+\ds\Delta t \,\left( [(\d U)_\imhalf^+  - (\d
          \tilde U)_\imhalf^+] \,\tilde \rho_{i-1} \,+\, [(\d U)_\imhalf^- - (\d \tilde U)_\imhalf^-] \,\tilde \rho_{i} \right),
\end{array}\right.
\end{equation}
where $A_i$, $B_{i-1}$, $C_{i+1}$ are nonnegative coefficients given by
$$
\left\{
\begin{array}{rl}
\ds A_i \!\!\!&=\,\ds +(\d V_1^n)_\ihalf^+- (\d V^n_1)_\imhalf^-  +  (\d
  U)_\ihalf^+ -  (\d U)_\imhalf^- + \frac{h'(\hat\rho_i)}{\Delta x_\ihalf}+ \frac{h'(\hat\rho_i)}{\Delta x_\imhalf},
\\[1.1em]
\ds B_{i-1} \!\!\!&=\,\ds  +(\d V_1^n)_\imhalf^+ +   (\d U)_\imhalf^+ + \frac{h'(\hat\rho_{i-1})}{\Delta x_\imhalf},
 \\[1.1em]
\ds C_{i+1} \!\!\!&=\, \ds -(\d V_1^n)_\ihalf^-- (\d U)_\ihalf^- + \frac{h'(\hat\rho_{i+1})}{\Delta x_\ihalf}.
\end{array}
\right.
$$
Now, we multiply  equation (\ref{tmp:0})  by ${\rm sign}(s_i)$ and sum
over $i=1,\ldots,N$, hence using that  $x\mapsto x^\pm$ is
Lipschitz continuous and observing that $A_i=B_i+C_i$, with $A_i,
\,B_i, \,C_i\,\geq\,
0$, it yields 
\begin{align*}
\sum_{i=1}^N \Delta x_i |s_i| \,\,\leq\,\,&  2\Delta t
                                            \,\|\tilde\rho\|_\infty\,\sum_{i=1}^N\left(
                                            \frac{|s_{i+1}|+|r_{i+1}|
                                          +|s_{i}|+|r_{i}|
                                            }{\Delta
                                            x_\ihalf} + \frac{|s_{i-1}|+|r_{i-1}|
                                          +|s_{i}|+|r_{i}|
                                            }{\Delta
                                            x_\imhalf} \right) 
\end{align*}
and in a similar way,
\begin{align*}
\sum_{i=1}^N \Delta x_i |r_i| \,\,\leq\,\,&  2\Delta t
                                            \,\|\tilde\eta\|_\infty \,\sum_{i=1}^N\left(
                                            \frac{|r_{i+1}|+|s_{i+1}|
                                          +|r_{i}|+|s_{i}|
                                            }{\Delta
                                            x_\ihalf} + \frac{|r_{i-1}|+|s_{i-1}|
                                          +|r_{i}|+|s_{i}|
                                            }{\Delta
                                            x_\imhalf} \right).
\end{align*}
Gathering these latter inequalities and from \eqref{admissible}, it
gives that 
\begin{align*}
\sum_{i=1}^N \Delta x_i \left( |s_i| + |r_i|\right) \,\,\leq\,\,16\,
  \|(\tilde\rho,\tilde\eta)\|_\infty \,\frac{\Delta t}{(\xi\,h)^2}\, \sum_{i=1}^N \Delta x_i \left( |s_i| + |r_i|\right). 
\end{align*}
Finally,  from the nonnegativity and the preservation of mass,  we
have
$$
\|(\tilde\rho,\tilde\eta)\|_\infty  \,\leq\,
\frac{m_1+m_2}{\xi h},
$$
hence under the condition \eqref{eq:CFL}, we conclude that $s=r=0$ and the uniqueness follows.
\end{proof}

It is worth to mention here that  the condition \eqref{eq:CFL} is not
optimal since we only use the discrete $L^1$-estimate on $\rho$ and $\eta$ to
control the discrete gradient and the $L^\infty$-norm.

We are now in position to state for  \eqref{eq:discrete_scheme_evol},
        \eqref{eq:implicit_numerical_fluxes} and
        \eqref{eq:implicit_num_noflux} an analogous results to the
semi-discrete case
\begin{theorem}[\label{thm:implicit_thm}Convergence to a weak solution of the implicit Euler discretisation]
Under the assumptions of Theorem \ref{thm:existence_discrete}, let $\rho_0, \eta_0 \in L_+^1(a,b) \cap L_+^\infty(a,b)$ be
        some initial data  and $Q_T := (0,T)\times (a,b)$ as
        above. Then, given two nonnegative sequences
        $(\rho_i^n)_{1\leq i \leq  N}$ and $(\eta_i^n)_{1\leq i\leq
          N}$ satisfying  \eqref{eq:discrete_scheme_evol},
        \eqref{eq:implicit_numerical_fluxes} and
        \eqref{eq:implicit_num_noflux}, for any $n\in\{0,\ldots,M\}$, then
\begin{itemize}
\item[$(i)$]  up to a subsequence, the piecewise constant 
        approximations converge strongly in $L^2(Q_T)$ to
        $(\rho,\eta) \in L^2(Q_T)$, where $(\rho,\eta)$ is  a weak solution as in
        Definition \ref{def:WeakSolution}. Furthermore we have $\rho$, $\eta
        \in L^2(0,T; H^1(a,b))$;
\item[$(ii)$] in particular, system \eqref{eq:crossdiffsystem} has a weak solution.
\end{itemize}
\end{theorem}

\begin{proof}[Sketch of the proof of Theorem \ref{thm:implicit_thm}]
	It is easily observed that the total mass is conserved due to the discrete no-flux boundary conditions, cf. \eqref{eq:implicit_num_noflux}. Together with the nonnegativity we were able to prove a semi-discrete version of the energy estimate Eq. \eqref{eq:cts_inequality} which is at the heart of the convergence result. Similarly as above, we are able to prove a fully discrete version of the energy estimate which then reads
	\begin{align*}
		\sum_{n=0}^M &\Delta t\sum_{i=1}^N \Delta x_i \frac{\rho_i^{n+1} \log \rho_i^{n+1} - \rho_i^n\log \rho_i^n}{\Delta t} +\frac{\eta_i^{n+1} \log \eta_i^{n+1} - \eta_i^n\log \eta_i^n}{\Delta t}\\
		 &+\,\sum_{n=0}^M\Delta t\sum_{i=1}^{N-1}\!\Delta x_{i+1/2}\left[\nu\, |\d U_\ihalf^{n+1}|^2 
		\,+\, \frac{\epsilon}{4}\left(  |\d \rho_\ihalf^{n+1}|^2
		\,+\,  |\d \eta_\ihalf^{n+1}|^2\right)
				\right]\,\leq\, C,
	\end{align*}
	with $C>0$ as in Corollary \ref{cor:apriori}. The inequality follows from the convexity of $x (\log x - 1)$ since
	\begin{align*}
		\rho_i^{n+1} (\log \rho_i^{n+1} - 1) - \rho_i^n (\log \rho_i^n -  1)\leq  \log(\rho_i^{n+1}) (\rho_i^{n+1} - \rho_i^n).
	\end{align*}
	Multiplying this expression by $\Delta x_i / \Delta t$ and summing over $i=1,\ldots N$ and $n=0,\ldots, M$ we obtain 
	\begin{align*}
		\sum_{n=0}^M\Delta t \sum_{i=1}^N\Delta x_i \frac{\rho_i^{n+1} (\log \rho_i^{n+1} - 1) - \rho_i^n (\log \rho_i^n -  1)}{\Delta t} \leq  \sum_{n=0}^M\Delta t \sum_{i=1}^N\Delta x_i\log(\rho_i^{n+1}) \frac{\rho_i^{n+1} - \rho_i^n}{\Delta t}.
	\end{align*}
	The right-hand side is then substituted by the scheme and simplified along the lines of the proof of Lemma \ref{lem:entropy_control}. Since the computations are exactly the same we omit them here for brevity and only note that it is important to set
	\begin{align}
	\tilde \rho_\ihalf^{n+1} := \left\{
		\begin{array}{ll}
			\ds\dfrac{\rho_{i+1}^{n+1} -
                  \rho_i^{n+1}}{\log\rho_{i+1}^{n+1} - \log
                  \rho_i^{n+1}}, & \text{if $\rho_i^{n+1} \neq
                                   \rho_{i+1}^{n+1}$},
\\ \, \\
			\ds\dfrac{\rho_i^{n+1} + \rho_{i+1}^{n+1}}{2}, & \text{else},
		\end{array}
		\right.
\end{align}
	to obtain the right sign in the numerical artefacts in Eq. \eqref{eq:numerical_artifacts}, which now read
$$
	\left\{
	\begin{array}{l}
	\displaystyle
	(\rho_i^{n+1} - \tilde \rho_\ihalf^{n+1}) \D\log
          \rho_\ihalf^{n+1} \big(\nu(\d U{n+1})_{\ihalf}^+ + (\d
          V_1^{n})_{\ihalf}^+\big) \leq 0,
\\ \, \\
	\displaystyle
	(\rho_{i+1}^{n+1} - \tilde \rho_\ihalf^{n+1}) \D\log \rho_\ihalf^{n+1} \big(\nu (\d U^{n+1})_{\ihalf}^- + (\d V_1^{n})_{\ihalf}^-\big) \leq 0.
	\end{array}
	\right.
$$
Following the lines of the proof Lemma \ref{lem:entropy_control} and Corollary \ref{cor:apriori} we obtain the fully discrete a priori bounds
	\begin{align*}
		\sum_{n=0}^N \Delta t \sum_{i=1}^{N-1} \Delta x_{\ihalf} \left(\,|\d_x\rho_\ihalf^{n+1}|^2 + |\d_x\eta_\ihalf^{n+1}|^2 + \,|\d_x U_\ihalf^{n+1}|^2 \right) \leq C,
	\end{align*}
	for some constant $C>0$, where we used `$\d_x$' to denote the discrete spatial gradient as before.
	Again, an application of Aubin-Lions Lemma provides relative compactness in the space $L^2((0,T)\times (a,b))$ of the piecewise constant interpolations $\rho_h, \eta_h$. As above, the discrete gradients are uniformly bounded in $L^2((0,T)\times (a,b))$ and their weak convergence is a consequence of the Banach-Alaoglu Theorem. Identifying the limits as a weak solution to system \eqref{eq:crossdiffsystem} is shown in the same way as above and we leave it as an exercise for the reader.
\end{proof}

\begin{rem}[Explicit Scheme]
	We do not consider an explicit scheme here since its analysis
        is much more complicated due to the lack of uniform
        estimates. Indeed, an explicit scheme  requires a CFL
        condition on the time step which appears in the stability
        analysis or the energy estimate provided in Lemma
        \ref{lem:entropy_control}. In our case, it leads to
$$
\frac{\rho^{n+1} \log \rho^{n+1} - \rho^n\log \rho^n}{\Delta
  t} = \frac{\rho^{n+1}  - \rho^n}{\Delta
  t} \left( \log \rho^n - 1\right) \,+\,   \frac{(\rho^{n+1}  - \rho^n)^2}{\left|\rho^{n+1/2}\right|^2\,\Delta
  t}, 
$$
where $\rho^{n+1/2}$ belongs to the interval $[\rho^n,
\rho^{n+1}]$ or  $[\rho^{n+1},
\rho^{n}]$. The control of this reminder term would require some
lower bounds estimates on the density $\rho$ (see for instance \cite{filbet-herda}). 

\end{rem}

\section{Numerical examples and validation\label{sec:numerical_results}}
In this section we perform some numerical simulations of system \eqref{eq:crossdiffsystem} using our scheme, Eqs. \eqref{eq:scheme}. In Section \ref{sec:error_numerical_convergence_order} we test our scheme by computing the error between the numerical simulation and a benchmark solution on a finer grid. Furthermore we determine the numerical convergence order. In Section \ref{sec:stationary_states} we compute the numerical stationary states of system \eqref{eq:crossdiffsystem} and discuss the implication of different cross-diffusivities and  the self-diffusivities, respectively.

Let us note here, that in the case of no regularising porous-medium
diffusion, \emph{i.e.} $\epsilon=0$, and certain singular potentials
the stationary states of system \eqref{eq:crossdiffsystem} are even
known  explicitly \cite{CHS17}. This allows us to compare the
numerical solution directly to the analytical stationary state in  Section \ref{sec:CHS17}.

Throughout the remainder of this section we apply the scheme \eqref{eq:scheme} to system \eqref{eq:crossdiffsystem}
using different self-diffusions, $\epsilon$,  and cross-diffusions, $\nu$.

\subsection{Error and numerical order of convergence\label{sec:error_numerical_convergence_order}}
This section is dedicated to validating our main convergence result, Theorem \ref{thm:convergence}. Due to the lack of  explicit solutions we compute the numerical solution on a fine grid and consider it  a benchmark solution. We then compute numerical approximations on coarser grids and study the error in order to obtain the numerical convergence order.

In all our simulations we use a fourth order Runge-Kutta scheme to solve the ordinary differential equations Eqs. \eqref{eq:scheme} with initial data Eqs. \eqref{eq:initial_data}. The discrepancy between the benchmark solution and numerical solutions on coarser grids is measured by the error
\begin{align*}
	e:=\left(\Delta t \sum_{k=1}^M \left(\Delta x \sum_{i=1}^N |\rho_{\mathrm{ex}}(t_k,x_i) - \rho(t_k, x_i)|^2 + |\eta_{\mathrm{ex}}(t_k,x_i) - \eta(t_k, x_i)|^2\right)\right)^{1/2}.
\end{align*}
Here $\rho_{\mathrm{ex}}$, $\eta_{\mathrm{ex}}$ denote the benchmark solutions. We use this quantity to study the convergence of our scheme as the grid size decreases. 
\subsubsection{No non-local interactions}
Let us begin with the purely diffusive system. We consider system \eqref{eq:crossdiffsystem}  without any interactions, \emph{i.e.}  $W_{ij}\equiv 0$, for $i,j=1,2$, and we choose $\nu=0.5$ and $\epsilon=0.1$. 

In Figure \ref{fig:purelydiffusive} we present the convergence result as we decrease the grid size. We computed a benchmark solution on a grid of $\Delta x = 2^{-10}$ on the time interval $[0,10]$ with $\Delta t = 0.05$. Figure \ref{fig:symmetric_initial_data_only_diffusion} shows the convergence for symmetric initial data whereas Figure \ref{fig:asymmetric_initial_data_only_diffusion} shows the convergence of the same system in case of asymmetric initial data. In both cases we overlay a line of slope one and we conclude that the numerical convergence is of order one.

\begin{figure}[ht!]
	\centering
	\subfigure[Symmetric initial data. $\rho_0(x)=\eta_0(x)= \mathds{1}_{[7,10]}$.]{
		\includegraphics[width=0.45\textwidth]{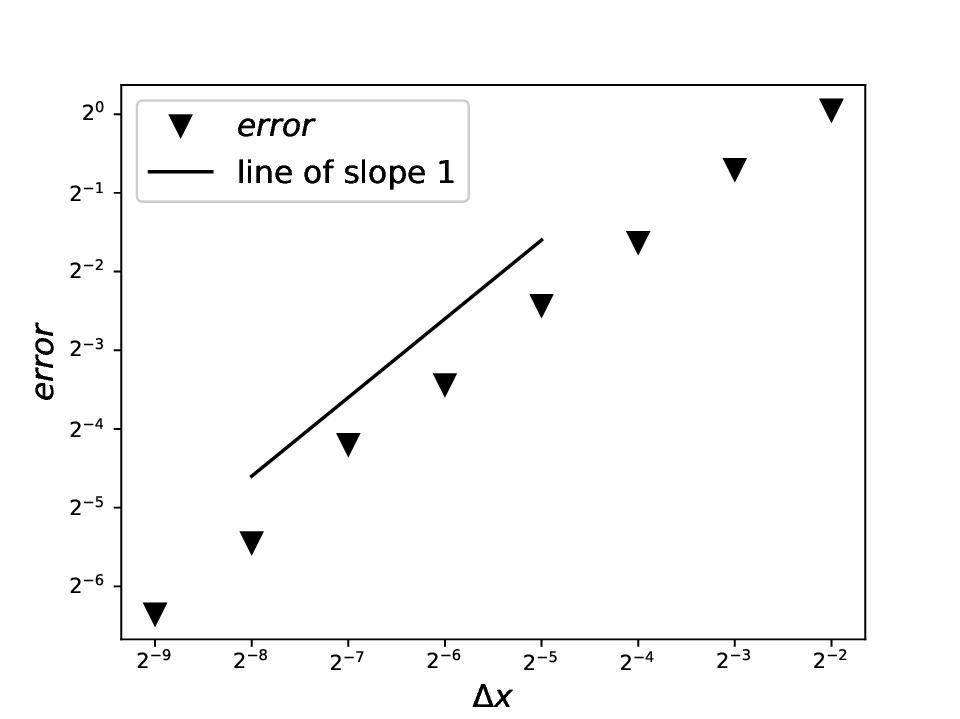}
		\label{fig:symmetric_initial_data_only_diffusion}
	}
	\hspace{1em}
	\subfigure[Asymmetric initial data. $\rho_0(x)=\mathds{1}_{[5,7]}$ and $\eta_0(x)=\mathds{1}_{[10,12]}$.]{
		\includegraphics[width=0.45\textwidth]{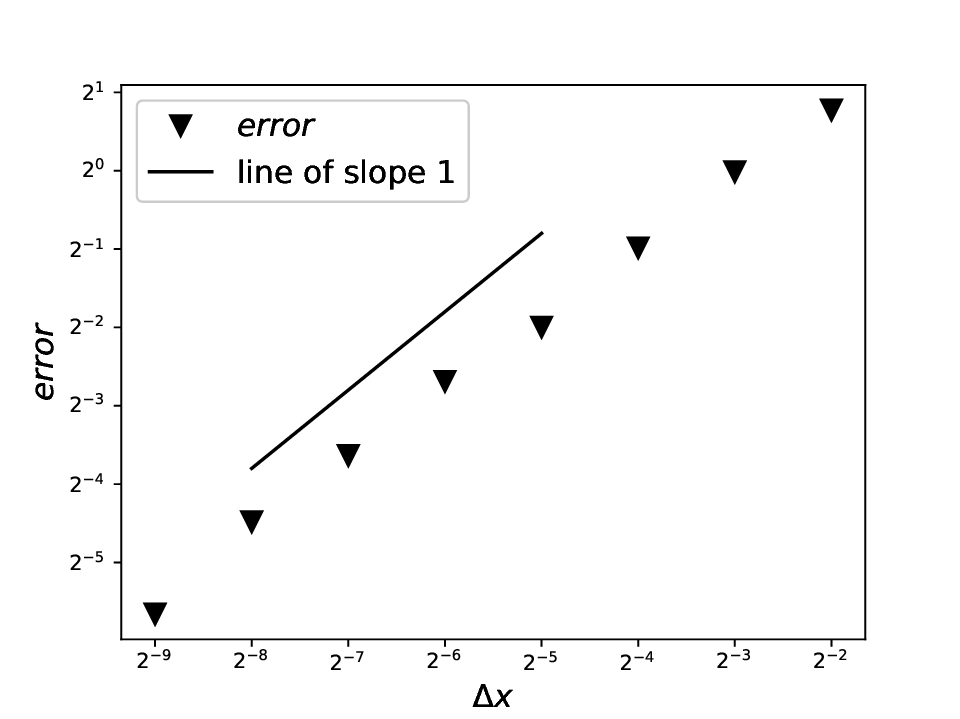}
	\label{fig:asymmetric_initial_data_only_diffusion}
	}
	\caption{In the purely diffusive system all interaction kernels are set to zero. Both graphs show the convergence to the benchmark solution. The triangular markers  denote the discrepancy between the numerical solution and the benchmark solution. A line of slope one is superimposed for the ease of comparison. On the left we start the system with symmetric initial conditions, on the right we start with asymmetric initial data. In both cases the scheme has numerical convergence order 1.}
	\label{fig:purelydiffusive}
\end{figure}

\subsubsection{Gaussian cross-interactions\label{sec:numeric_gaussians}}
Next, we add non-local self-interaction and cross-interactions. We choose smooth Gaussians with different variances. These potentials, like the related, more singular Morse potentials, are classical in mathematical biology since oftentimes the availability of sensory information such as sight, smell or hearing is spatially limited \cite{OCBC06, CMP13, CHM14}. For the intraspecific interaction we use
\begin{align*}
	W_{11}(x) = W_{22}(x) = 1 - \exp\left(-\frac{|x|^4}{4\times 0.1}\right),
\end{align*}
while we choose
\begin{align*}
	W_{12}(x) = -W_{21}(x) = 1 - \exp\left(-\frac{|x|^2}{2\times 0.1}\right),
\end{align*}
for the interspecific interactions.

We consider system \eqref{eq:crossdiffsystem} with the diffusive coefficients $\nu = 0.4$ and $\epsilon\in\{0.1, 0.5\}$, and we initialise the system with
\begin{align*}
	\rho(x) = \eta(x) = c \big((s-6.5)(9.5-s)\big)^+,
\end{align*}
on the domain $[0,9]$. Here the constant $c$ is such that $\rho$ and $\eta$ have unit mass. Figure \ref{fig:gaussian_interactions_error} depicts the simulation with Gaussian kernels as interaction potentials. In Figures \ref{fig:gaussian_interactions_low_self_diffusion} \& \ref{fig:gaussian_interactions_high_self_diffusion} we present the error plots  corresponding to $\epsilon\in\{0.1,0.5\}$. Again we observe convergence to the reference solution with a first order accuracy in space. 
\begin{figure}[ht!]
	\centering
	\subfigure[Convergence to benchmark solution for individual diffusion constant $\epsilon=0.1$.]{
		\includegraphics[width=0.45\textwidth]{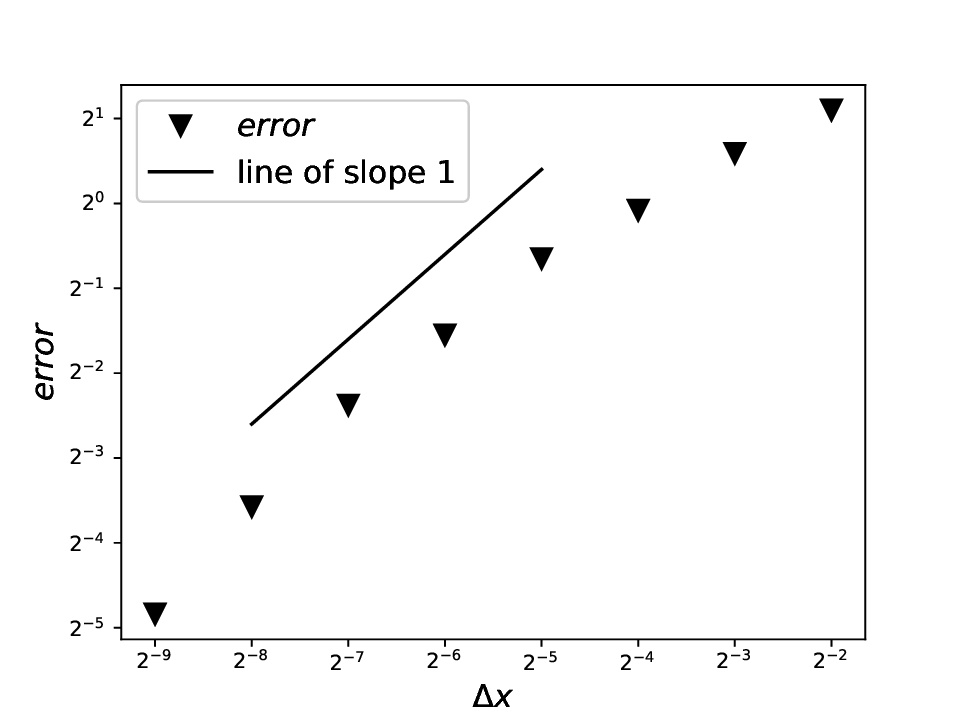}
		\label{fig:gaussian_interactions_low_self_diffusion}
	}
	\hspace{1em}
	\subfigure[Convergence to benchmark solution for individual diffusion constant $\epsilon=0.5$.]{
		\includegraphics[width=0.45\textwidth]{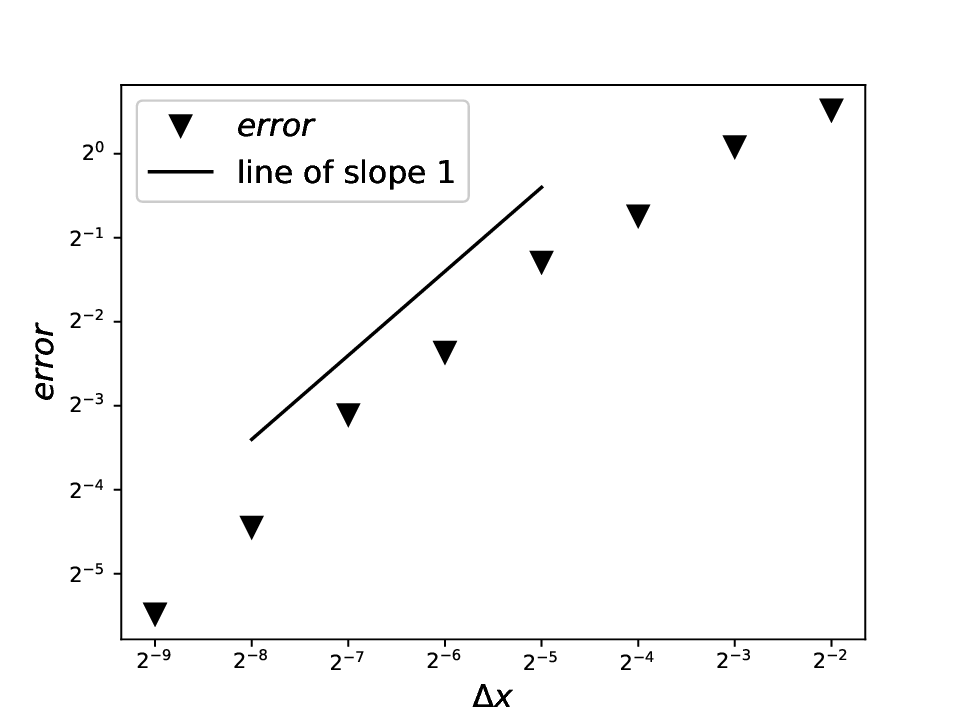}
		\label{fig:gaussian_interactions_high_self_diffusion}
	}
	\caption{We choose Gaussian interaction kernels of different strengths and ranges for the self-interaction and the cross-interaction, respectively. The graphs show the numerical convergence order in the cases of $\epsilon=0.1$ (left),  and  $\epsilon=0.5$ (right), respectively.}
	\label{fig:gaussian_interactions_error}
\end{figure}

\subsection{General behaviour of solutions and stationary states\label{sec:stationary_states}}
In this section we aim to study the asymptotic behaviour of system \eqref{eq:crossdiffsystem} numerically.  Let us begin by going back to the set up of Section \ref{sec:numeric_gaussians}. We note that the potentials were chosen in such a way that there is an attractive intraspecific force, and the cross-interactions are chosen as attractive-repulsive explaining the segregation observed in Figure \ref{fig:gaussian_interactions_left}. For larger self-diffusivity, $\epsilon=0.5$, we see that some mixing occurs. In the absence of any individual diffusion we would have expected adjacent species with a jump discontinuity at their shared boundary \cite{CHS17}. However, this phenomenon is no longer possible as we have a control on the gradients of each individual species, by Lemma \ref{cor:apriori}, rendering jumps impossible thus explaining the continuous transition. 
\begin{figure}[ht!]
	\centering
	\subfigure[Stationary state ($\epsilon=0.1$).]{
		\includegraphics[width=0.47\textwidth]{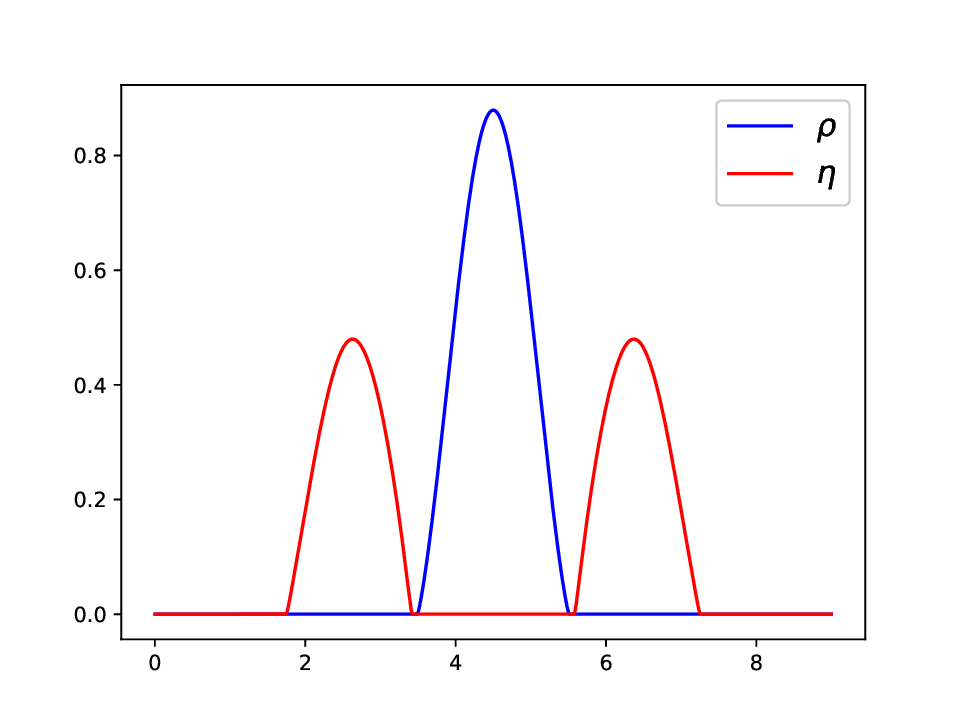}
		\label{fig:gaussian_interactions_left}
	}
	\subfigure[Stationary state ($\epsilon=0.5$).]{
		\includegraphics[width=0.47\textwidth]{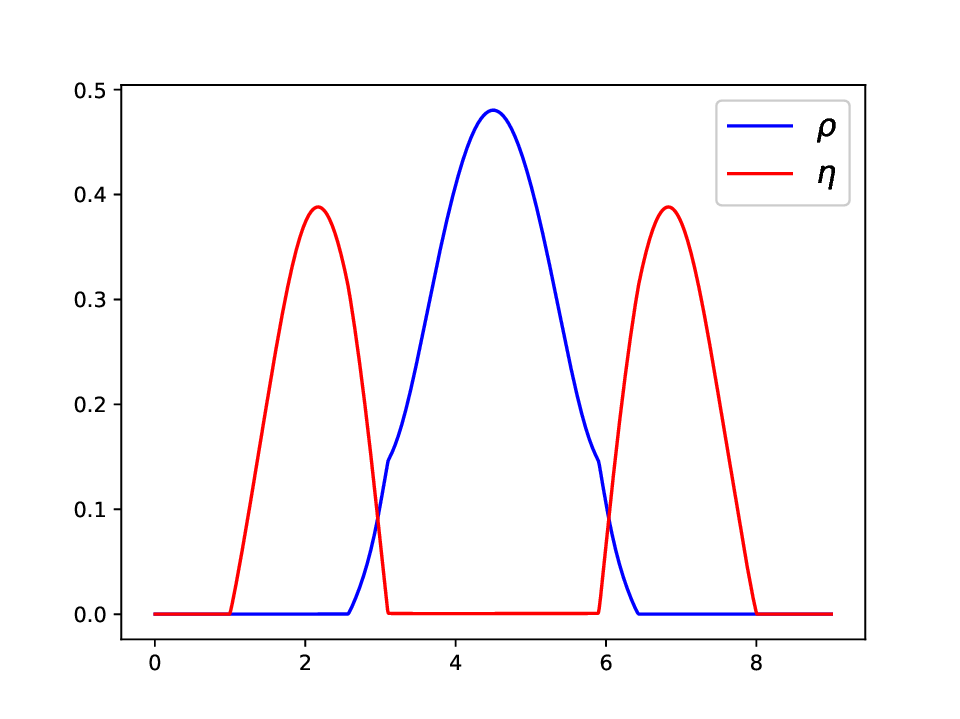}
		\label{fig:gaussian_interactions_right}
	}
	\caption{We choose Gaussian interaction kernels of different strengths and ranges for the self-interaction and the cross-interaction, respectively. The graphs correspond to the simulationed stationary states for self-diffusivities $\epsilon=0.1$, and $\epsilon=0.5$, respectively.}
	\label{fig:gaussian_interactions_steady_states}
\end{figure}

In the subsequent section we shall push our scheme even further by dropping the smoothness assumption on our potentials.

\subsubsection{\label{sec:CHS17_regularised}\label{sec:CHS17}Case of singular potentials}
In this section we go beyond the limit of what we could prove in this paper. On the one hand we consider more singular potentials and on the other hand we consider vanishing individual diffusion. We study system \eqref{eq:crossdiffsystem} for $\epsilon\in\{0,\,0.02,\,0.04,\, 0.06,\, 0.09\}$, and $\nu \in\{0.05, 0.5\}$. 
Here the potentials are given by
\begin{align*}
	W_{11}(x) = W_{22}(x) = x^2/2.
\end{align*}
for the self-interaction terms and  
\begin{align*}
	W_{12}(x) = |x| = \pm W_{21}(x).
\end{align*}
for the cross-interactions. The system is posed on the domain $[0,5]$
with a grid size of $\Delta x = 2^{-8}$. Note that the case of
$\epsilon=0$ corresponds to the absence of individual diffusion,  see
Figures \ref{fig:CHS_agreement_attrrep} \&
\ref{fig:CHS_agreement_attrrep_adjacent}. By virtue of Corollary
\ref{cor:apriori}, it is the individual diffusion that regularises the
stationary states, in the sense that we will not observe any
discontinuities in either $\rho$ or $\eta$. As we add individual
diffusion we can see the immediate regularisation. While stationary
states may still remain segregated, as it is shown in Figure
\ref{fig:CHS_agreement_attrrep_increasing_eps}, adjacent solutions are
not possible anymore (see Figure \ref{fig:CHS_agreement_attrrep_increasing_eps_adjacent}).
\begin{figure}[ht!]
	\centering
	\subfigure[In the case $\nu = 0.05$, $\epsilon=0$, we obtain a great agreement of the numerically computed stationary states and the analytical stationary states described in \cite{CHS17}.]{
	\includegraphics[width=0.47\textwidth]{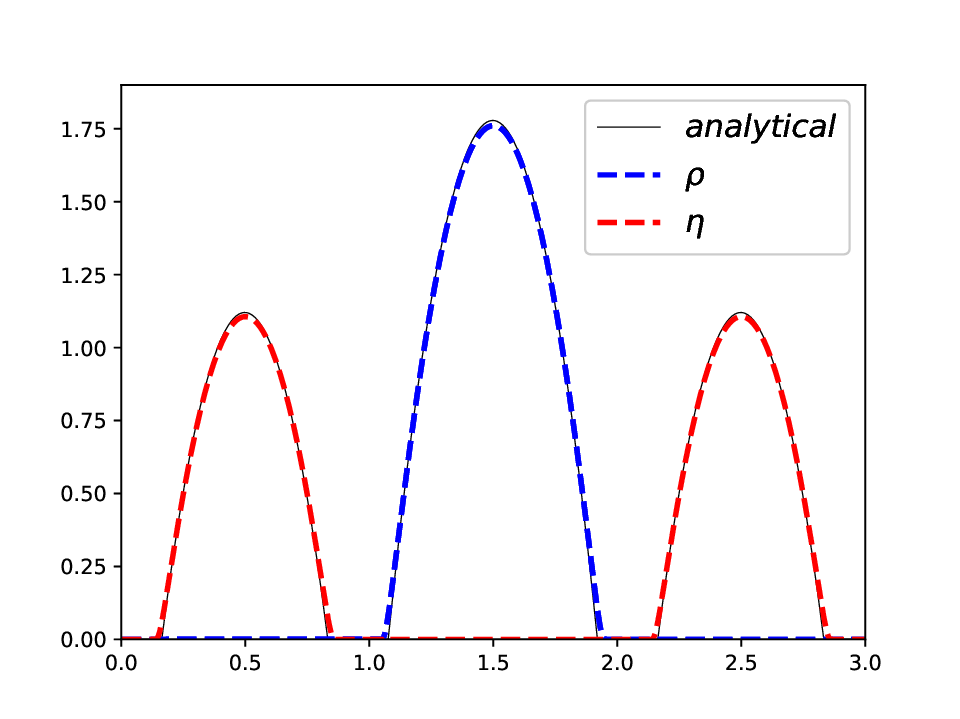}
	\label{fig:CHS_agreement_attrrep}
	}\hspace{0.2cm}
	\subfigure[Adding individual diffusion may still lead to segregated stationary states. However both species remain continuous as they mix.]{
	\includegraphics[width=0.47\textwidth]{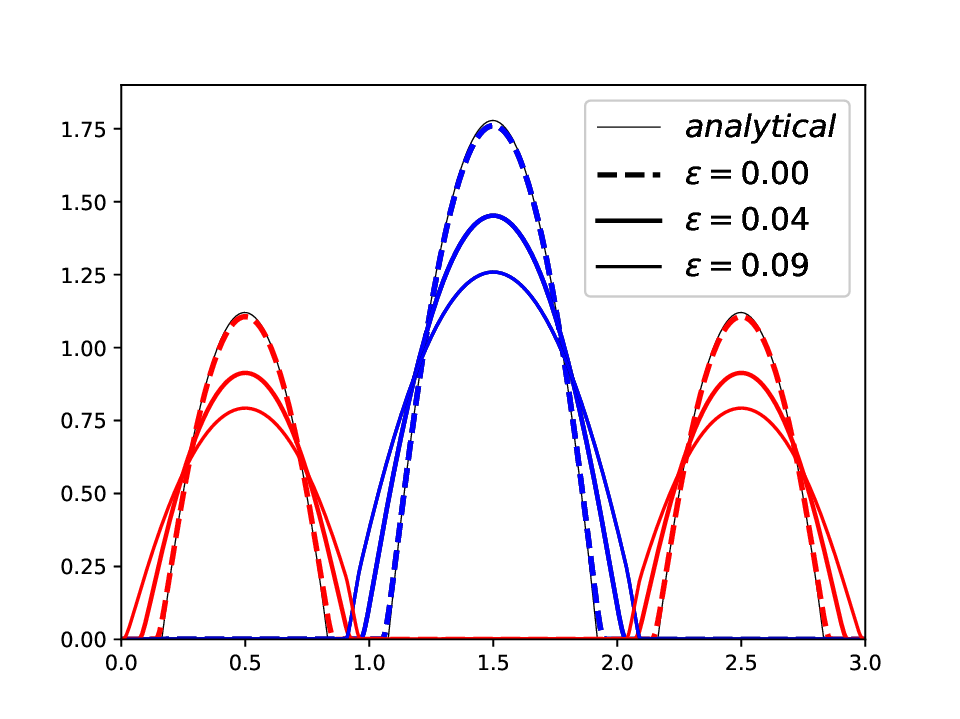}
	\label{fig:CHS_agreement_attrrep_increasing_eps}
	}
	\subfigure[The case $\nu = 0.5$, $\epsilon=0$ leads to adjacent stationary states. Again we see an excellent agreement of the numerical stationary states and the analytical ones \cite{CHS17}.]{
	\includegraphics[width=0.47\textwidth]{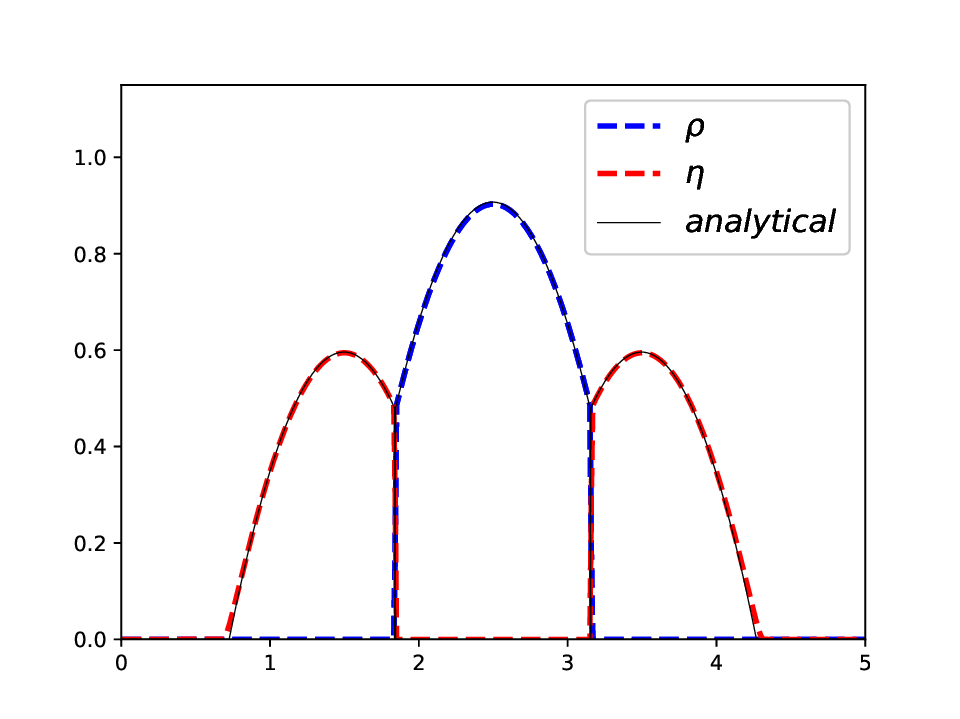}
	\label{fig:CHS_agreement_attrrep_adjacent}
	}\hspace{0.2cm}
	\subfigure[The regularising effect of the individual
        diffusion, by Corollary \ref{cor:apriori}, becomes apparent immediately. Instantaneously both species become continuous as they start to intermingle in a small region. This region grows as we keep increasing the individual diffusion.]{
	\includegraphics[width=0.47\textwidth]{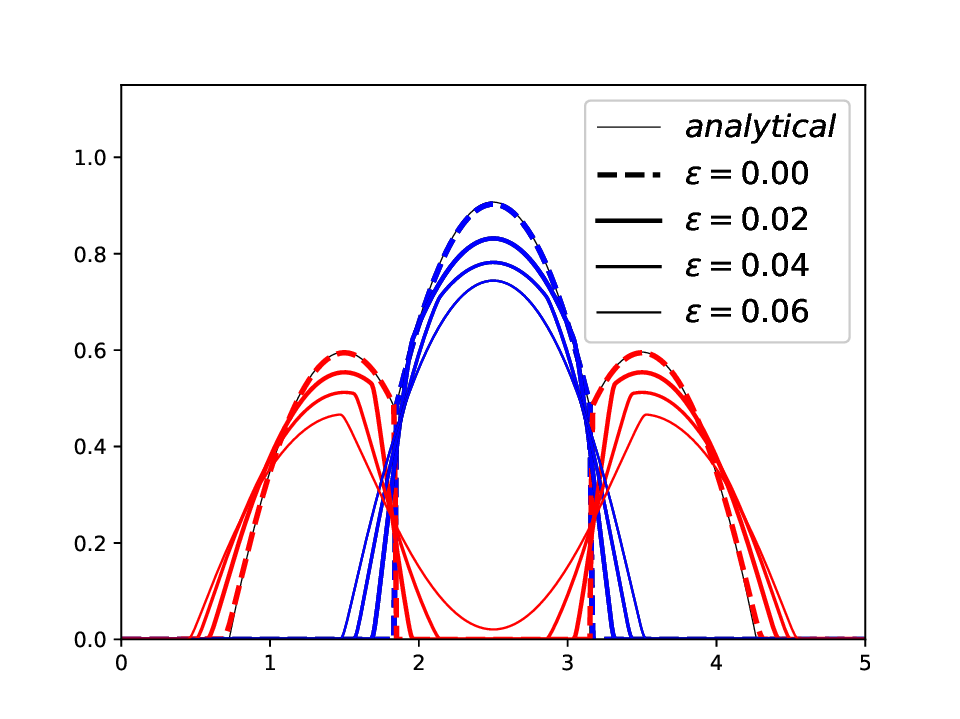}
	\label{fig:CHS_agreement_attrrep_increasing_eps_adjacent}
	}	
	\caption{We pushed our numerical scheme to see how it performs in regimes in which we are unable to prove convergence. We chose Newtonian cross-interactions in the attractive-repulsive case. The red curves denote the symmetric stationary states of $\eta$ while the blue curves are the stationary distributions of $\rho$. The different line widths and styles correspond to varying values of $\epsilon$.}
\end{figure}
In the case of attractive-repulsive interspecific interactions, \emph{i.e.}
\begin{align*}
	W_{12}(x) = |x| = -W_{21}(x).
\end{align*}
 we expect both species to segregate \cite{CHS17}. We initialise the system with the following symmetric initial data
\begin{align*}
	\rho(x) = \eta(x) = c \big((x-3)(5-x)\big)^+,	
\end{align*}
as symmetric initial data are known to approach stationary states \cite{CHS17}. Here the constant $c$ normalises the mass of $\rho$ and $\eta$ to one.

Figures \ref{fig:CHS_agreement_attrrep} \& \ref{fig:CHS_agreement_attrrep_adjacent} show our scheme performs well even in regimes we are unable to show convergence due to the lack of regularity in the potentials as well as the lack of regularity due to the absence of the porous medium type self-diffusion. While the schemes developed in \cite{CCH15, CHS17} are asymptotic preserving, their convergence to weak solutions of the respective equations could not be established. We reproduce the steady states of \cite{CHS17} that exhibit phase separation phenomena. Figure \ref{fig:additional_numerics} displays the stationary state in the case $\nu=0.09$ and $\epsilon=0$. Even in the case of no regularising individual diffusion and with Newtonian cross-interactions we observe a numerical convergence order of one.

\begin{figure}[ht!]
	\centering
	\subfigure[Strictly segregated stationary state in the absence of individual diffusion, $\epsilon=0$, and Newtonian potentials.]{
	\includegraphics[width=0.47\textwidth]{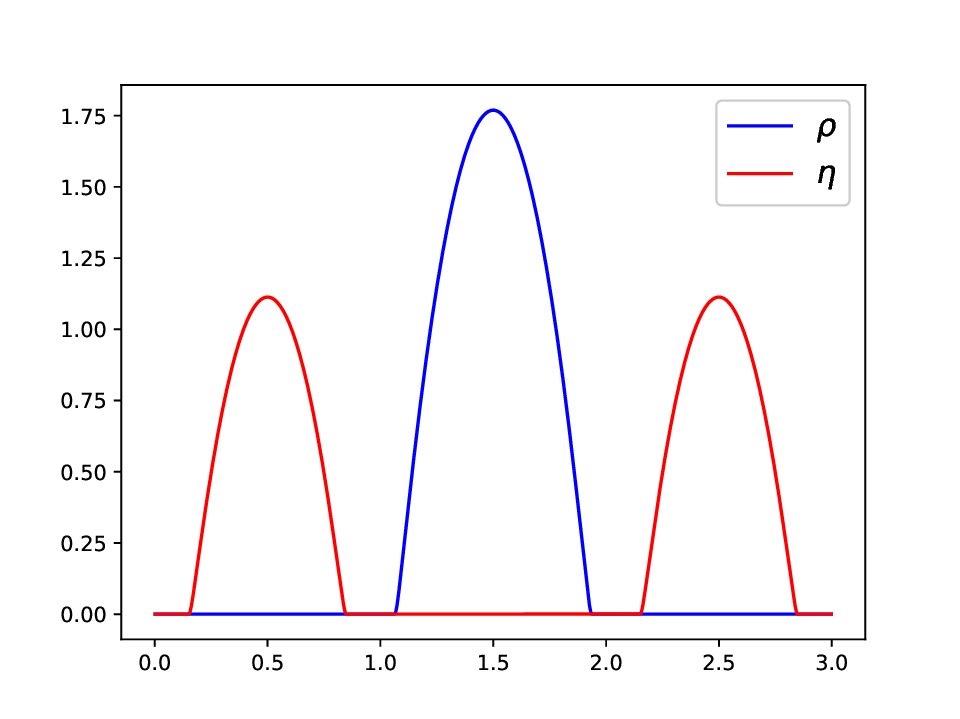}
	}
	\subfigure[Convergence to benchmark solution in the case of $\epsilon=0$ and Newtonian attractive-repulsive potentials.]{
	\includegraphics[width=0.47\textwidth]{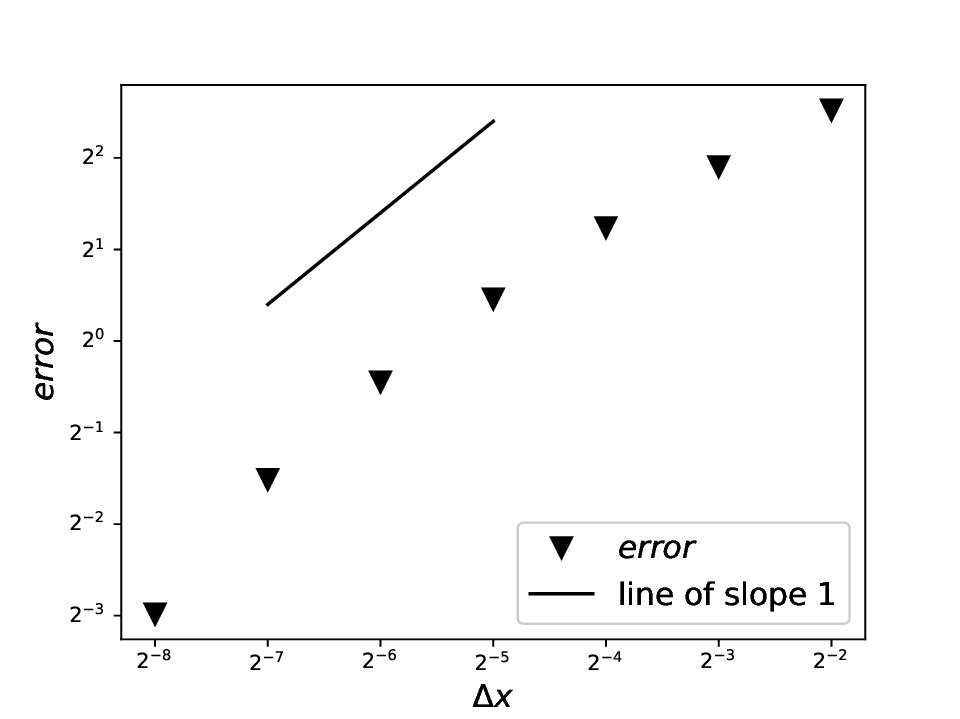}
	}
	\caption{Stationary state and numerical convergence order in the case of Newtonian attractive-repulsive cross-interactions. Even though our estimates fail in the analysis above we observe a numerical convergence order of one.}
	\label{fig:additional_numerics}
\end{figure}

In the case of attractive-attractive cross-interactions, \emph{i.e.} $W_{12}(x) = |x| = W_{21}(x)$, we observe an interesting phenomenon. Even in the absence of the individual diffusion, \emph{i.e.} $\epsilon=0$, some additional mixing occurs even though we expect sharp boundaries, see Figure \ref{fig:CHS_agreement_attrattr}, due to numerical diffusion. This is in contrast to the finite volume schemes proposed in \cite{CCH15, CHS17}.

\begin{figure}[ht!]
	\centering
\includegraphics[width=0.47\textwidth]{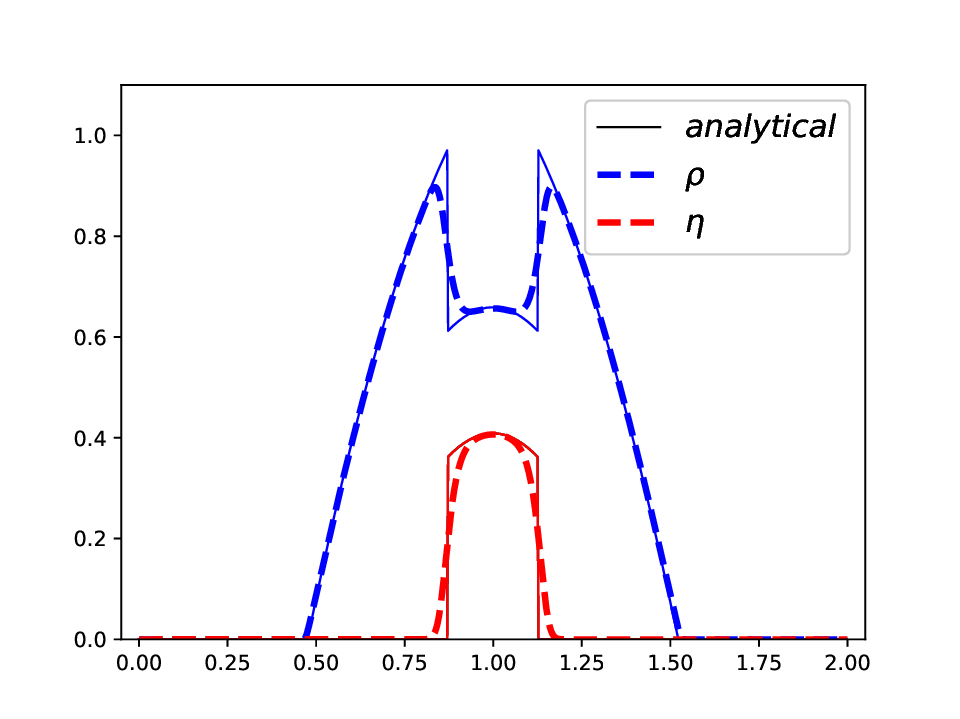}
	\caption{We choose $m_1 = 0.6$ and $m_2=0.1$ in order to be able to compare the stationary state with the explicit one given in \cite{CHS17}. We can see a strong resemblance between the numerical stationary state and the one obtained analytically. However there are some regimes of mixing due to numerical diffusion.}
\label{fig:CHS_agreement_attrattr}
\end{figure}

\subsection{Energy dissipation}
It is known that system \eqref{eq:crossdiffsystem} has a formal gradient flow structure, \emph{cf.} \cite{DFEF17}, whenever $W_{12} = W_{21}$. In this case, the evolution of system \ref{eq:crossdiffsystem} is such that it decays the energy functional
\begin{align*}
	\frak{E}(\rho, \eta) := \frac12 \iint \rho W_{11}\star \rho + \eta W_{22}\eta\, \d x + \iint \rho W_{12} \star \eta\, \d x + \frac12 \int \nu (\rho+\eta)^2 + \epsilon \rho^2 + \epsilon \eta^2 \, \d x.
\end{align*}

Here, we present two examples, one corresponding to the potential
\begin{align*}
	W_{ii}(x) = \frac{x^2}{2}, \quad \text{and} \quad W_{ij}(x) = |x|,
\end{align*}
for $i,j=1,2$ and $i\neq j$,
\emph{cf.} Figure \ref{fig:parabolamodulus}, the other one corresponding to
\begin{align*}
	W_{ii}(x) = 1 - \exp\left(-\frac{(4x)^2}{2}\right), \quad \text{and}\quad W_{ij}(x)= 1 + \exp\left(-(4x)^2\right) - \exp\left(-\frac23  (4x)^2\right),
\end{align*}
for $i,j=1,2$ and $i \neq j$, \emph{cf.} Figure \ref{fig:gaussianpotentials}.

In the first case, we choose the initial data 
\begin{align*}
	\rho(x) = c \big((x-2)(2.5-x)\big)^+,	\quad \text{and} \quad	\eta(x) = c \big((x-1)(0.5-x)\big)^+,
\end{align*}
where $c>0$ normalises the mass to one. Due to the long-range we observe an attraction of the two initial bumps until they meet. They begin to mix until they are completely merged. The graph in the last panel shows the decay of the associated energy to a constant one which corresponds to the energy of the steady state. It appears the energy is dissipated at an exponential rate but an analytic result for systems, corresponding to that of a single equation, \emph{cf.} \cite{CMcCV03}, is not known to our knowledge.
\begin{figure}
	\centering
	\subfigure[$t=0$]{
	\includegraphics[width=0.31\textwidth]{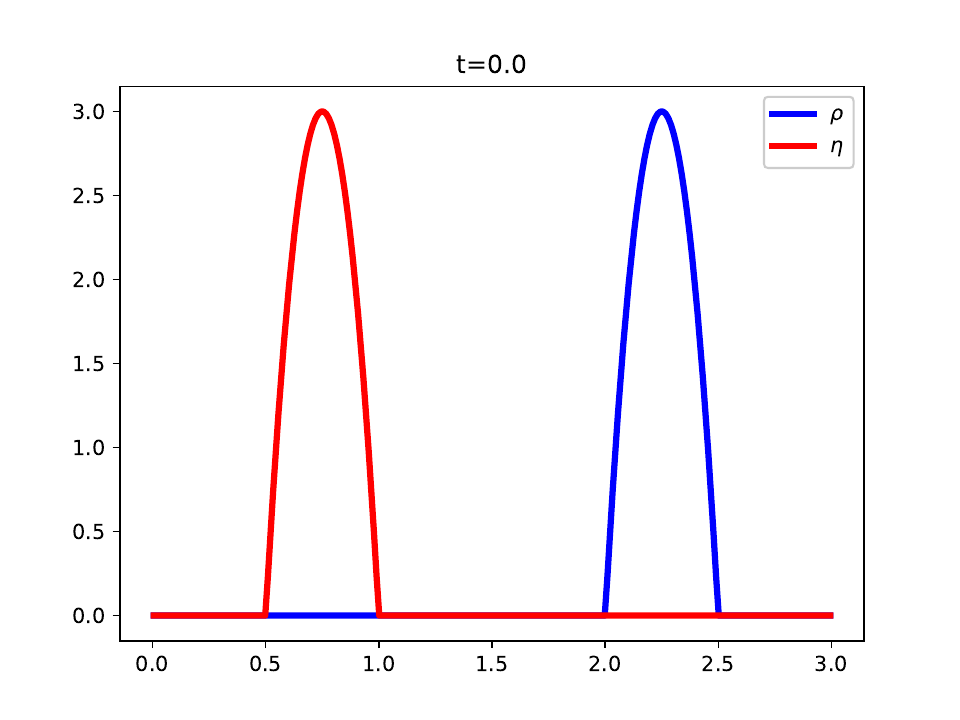}
	}
	\subfigure[$t=0.211$]{
	\includegraphics[width=0.31\textwidth]{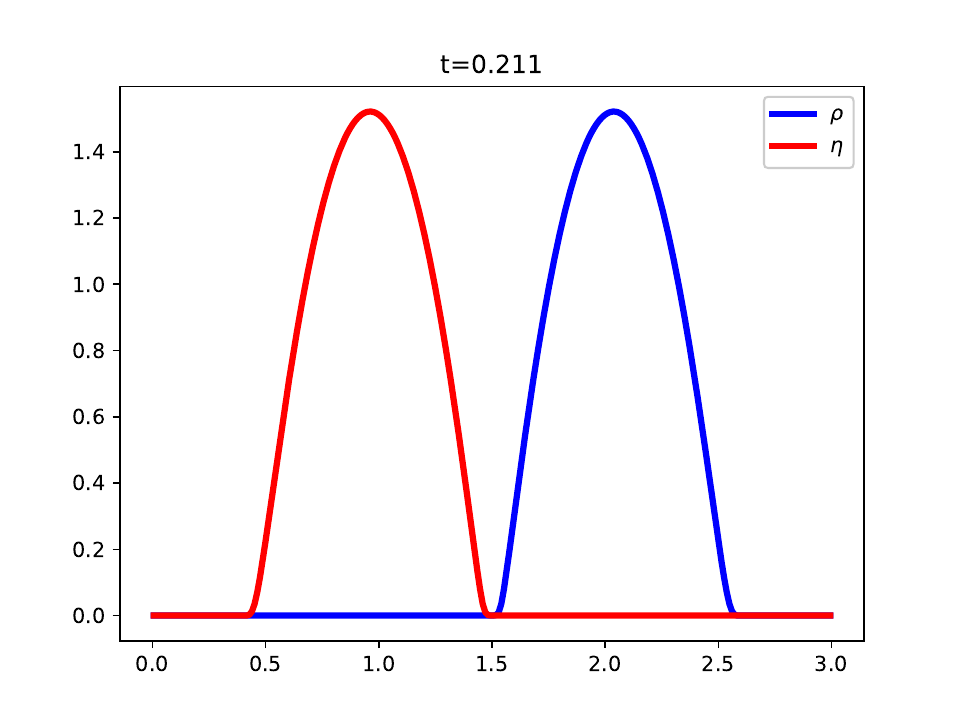}
	}
	\subfigure[$t=0.433$]{
	\includegraphics[width=0.31\textwidth]{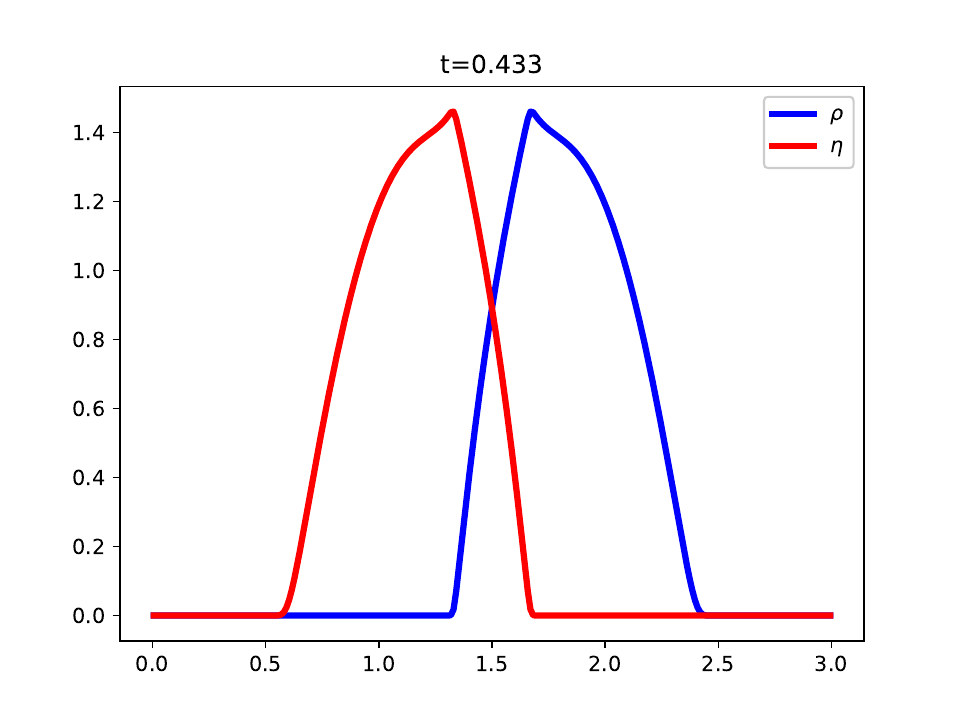}
	}
	\subfigure[$t=0.614$]{
	\includegraphics[width=0.31\textwidth]{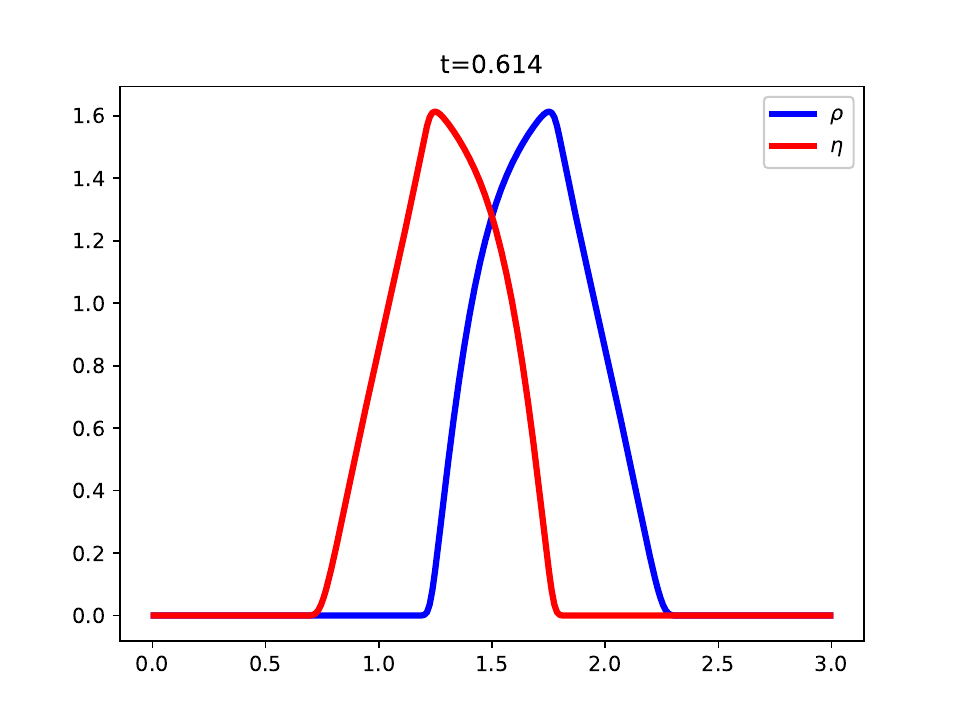}
	}
	\subfigure[$t=0.768$]{
	\includegraphics[width=0.31\textwidth]{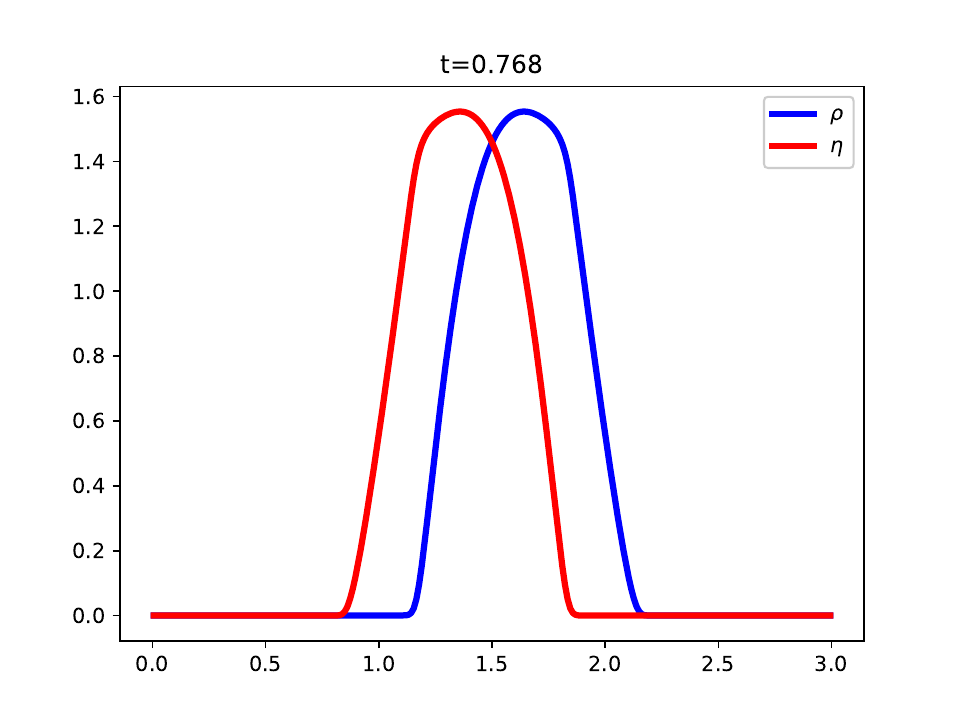}
	}
	\subfigure[$t=0.909$]{
	\includegraphics[width=0.31\textwidth]{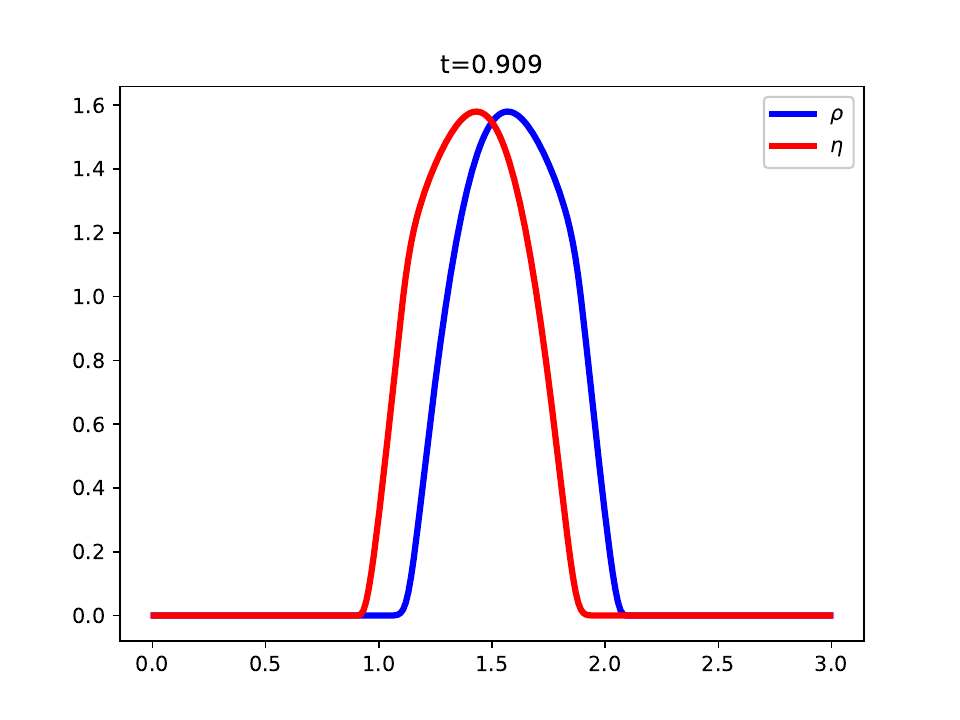}
	}
	\subfigure[$t=1.111$]{
	\includegraphics[width=0.31\textwidth]{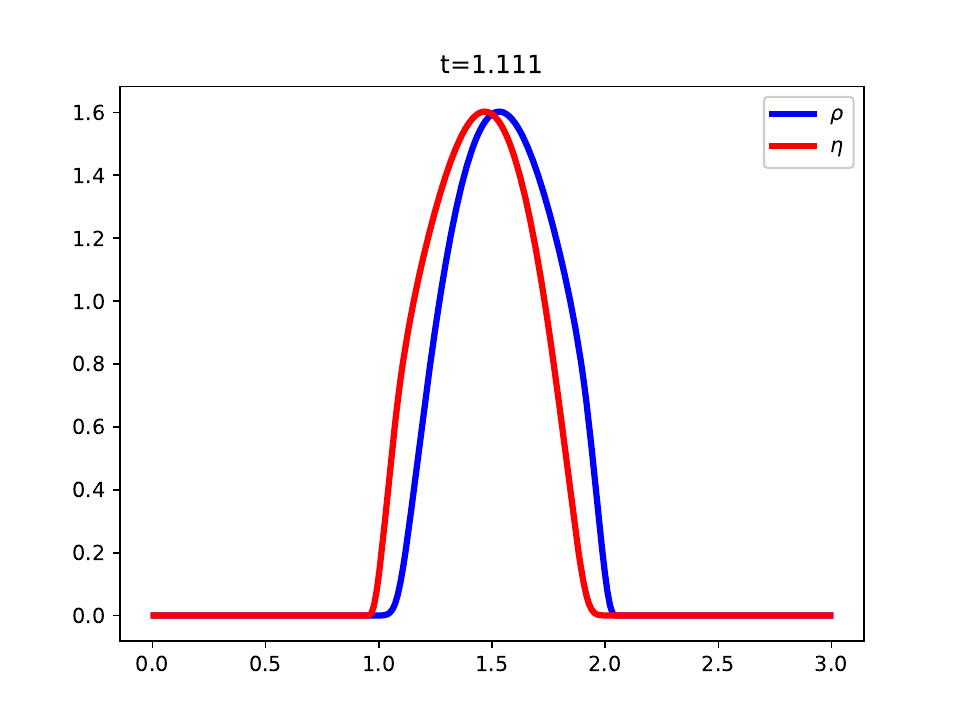}
	}
	\subfigure[$t = 1.41$]{
	\includegraphics[width=0.31\textwidth]{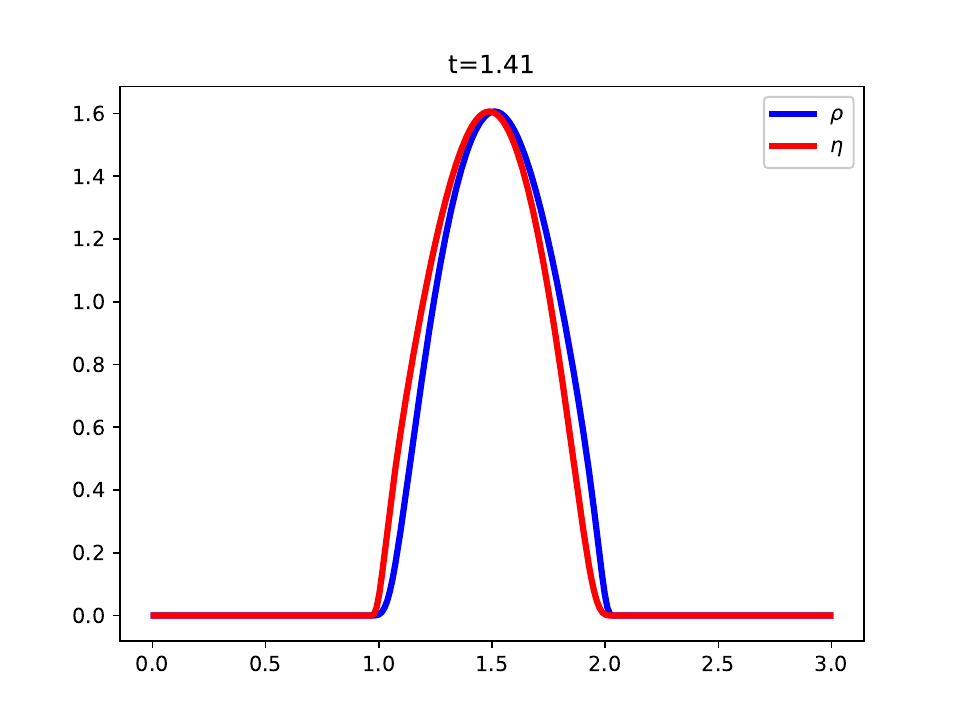}
	}
	\subfigure[Decay of the energy, $\frak{E}(\rho, \eta)$.]{
	\includegraphics[width=0.31\textwidth]{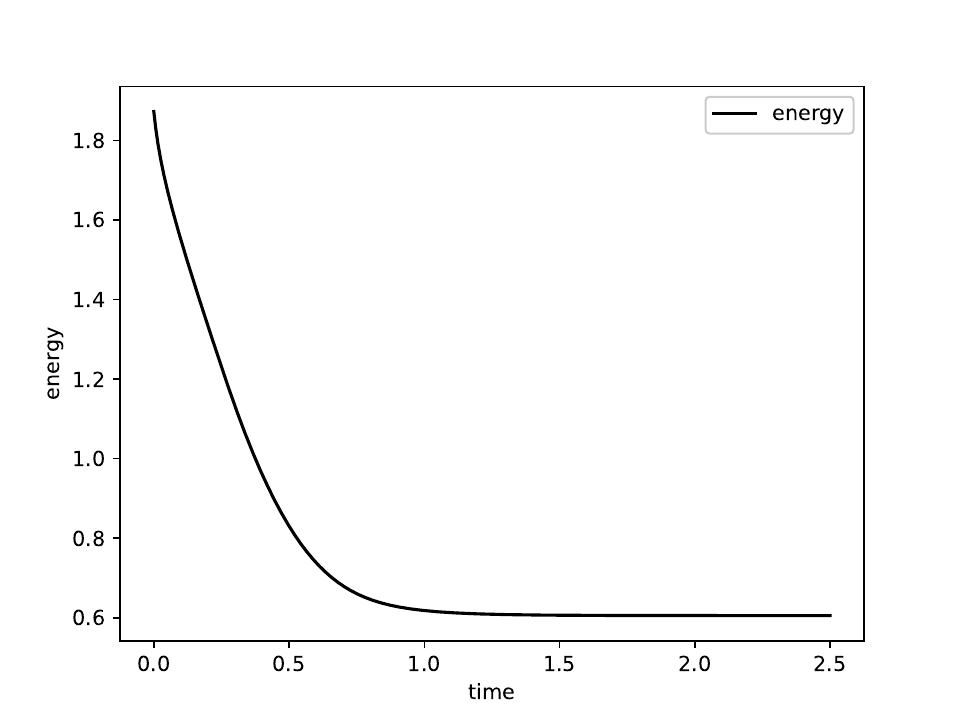}
	}
	\caption{Evolution of segregated initial data for attractive-attractive interactions with corresponding potentials $W_{ii}(x) = x^2/2$ and $W_{ij}(x) = |x|$.  The two blobs move towards each other. The associated energy appears to converge exponentially fast to a constant while the profiles merge.}
	\label{fig:parabolamodulus}
\end{figure}

In the second case, for Gaussian potentials, we change the computational domain to $(0,\pi)$, for convenience. We choose the initial data
\begin{align*}
	\rho(x) = \sin(2x)^2, \quad \text{and} \quad \cos(2x)^2,
\end{align*}
\emph{cf.} Figure \ref{fig:gaussianpotentials}. We observe the formation of nearly segregated clusters which, as the evolution continues, as begin to merge due to the nonlocal interaction. However, the short-range cross-interaction is working against this trend which explains that the evolution slows down just before the merging, a phenomenon which is also observed in meta-stability. After the 5 clusters have merged into three the profile stabilises which is reflected in the evolution of the energy, \emph{cf.} graph in the panel. We still observe a decay, however, after a strong initial decrease the energy decays much slower for a while before going to the constant corresponding to the stationary state. The explanation lies in the increase of the internal energy. Initially, we have $\rho + \eta \equiv 1$ which is a minimiser of the internal energy. Due to the nonlocal interactions the system wants to rearrange but at the cost of increasing the internal energy to allow for a decrease in the interaction energy. This beautifully portrays the interplay of local and nonlocal effects. Similar effects are known in the context of meta-stability. 

\begin{figure}
	\centering
	\subfigure[$t=0.0$]{
	\includegraphics[width=0.31\textwidth]{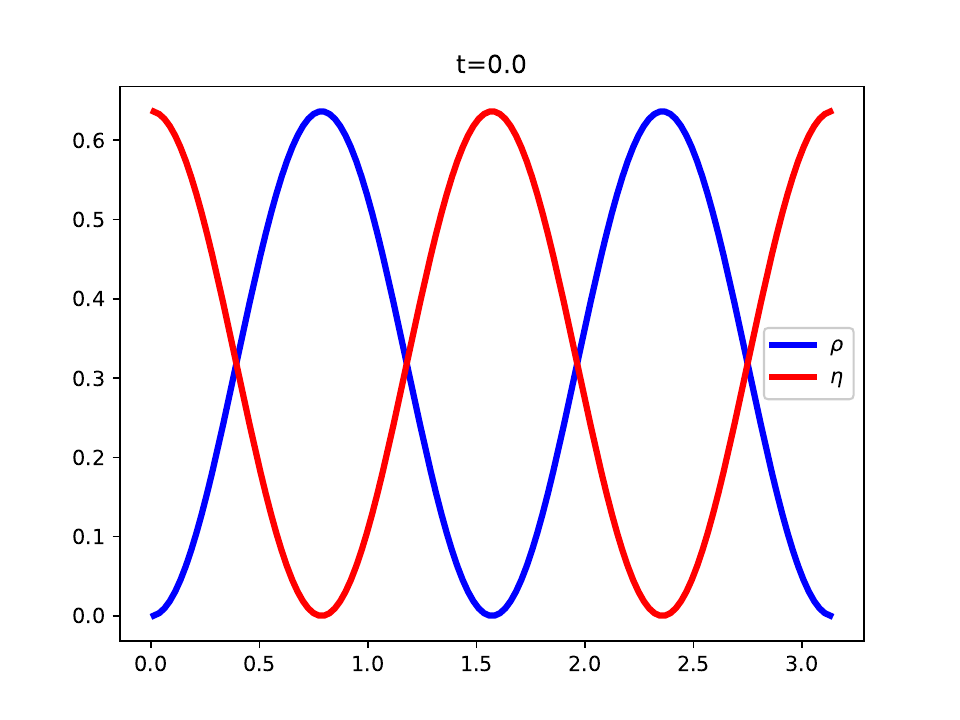}
	}
	\subfigure[$t=1.05$]{
	\includegraphics[width=0.31\textwidth]{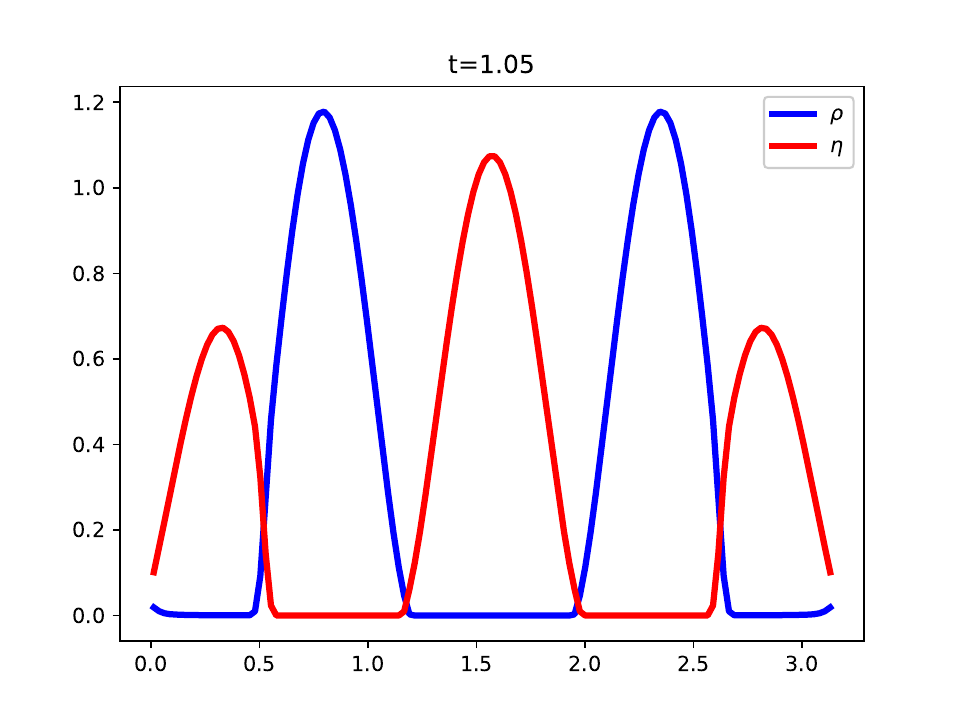}
	}
	\subfigure[$t=5.396$]{
	\includegraphics[width=0.31\textwidth]{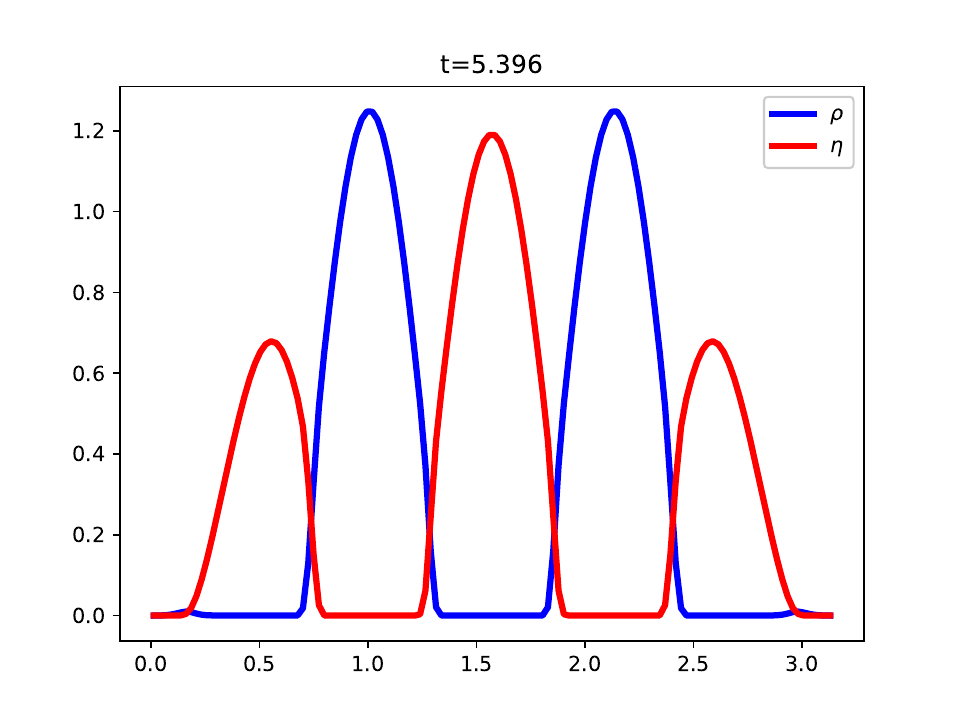}
	}
	\subfigure[$t=9.89$]{
	\includegraphics[width=0.31\textwidth]{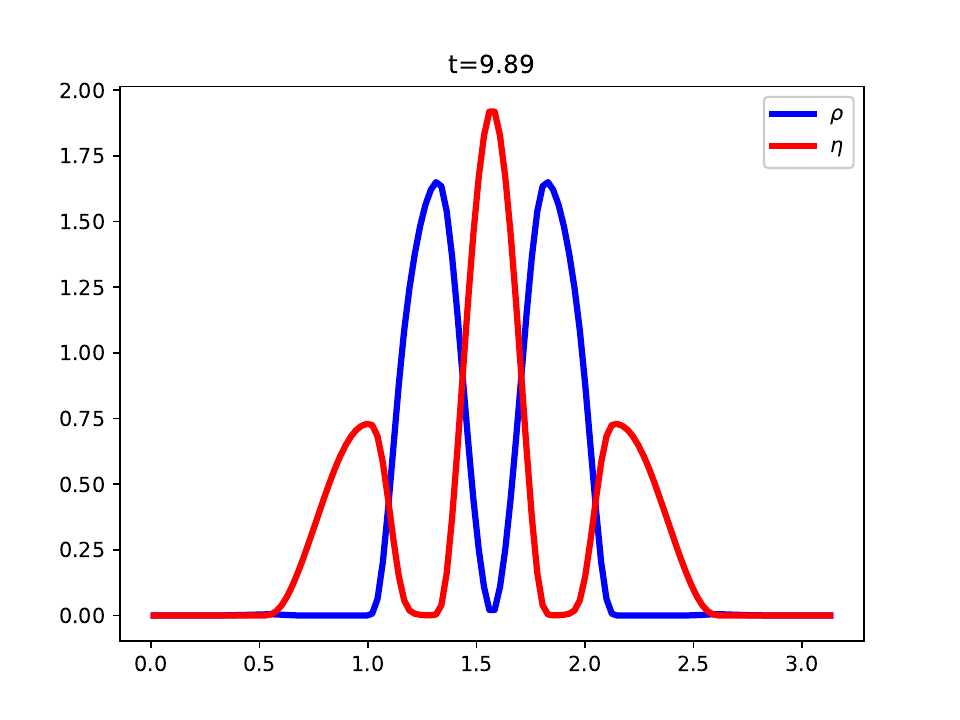}
	}
	\subfigure[$t=10.417$]{
	\includegraphics[width=0.31\textwidth]{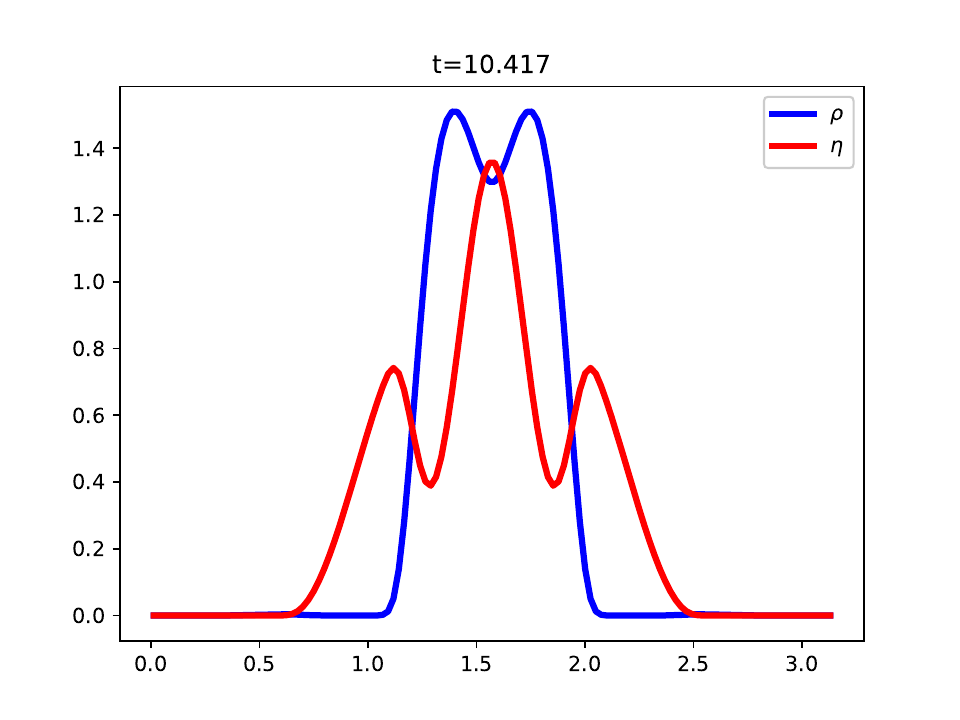}
	}
	\subfigure[$t=10.58$]{
	\includegraphics[width=0.31\textwidth]{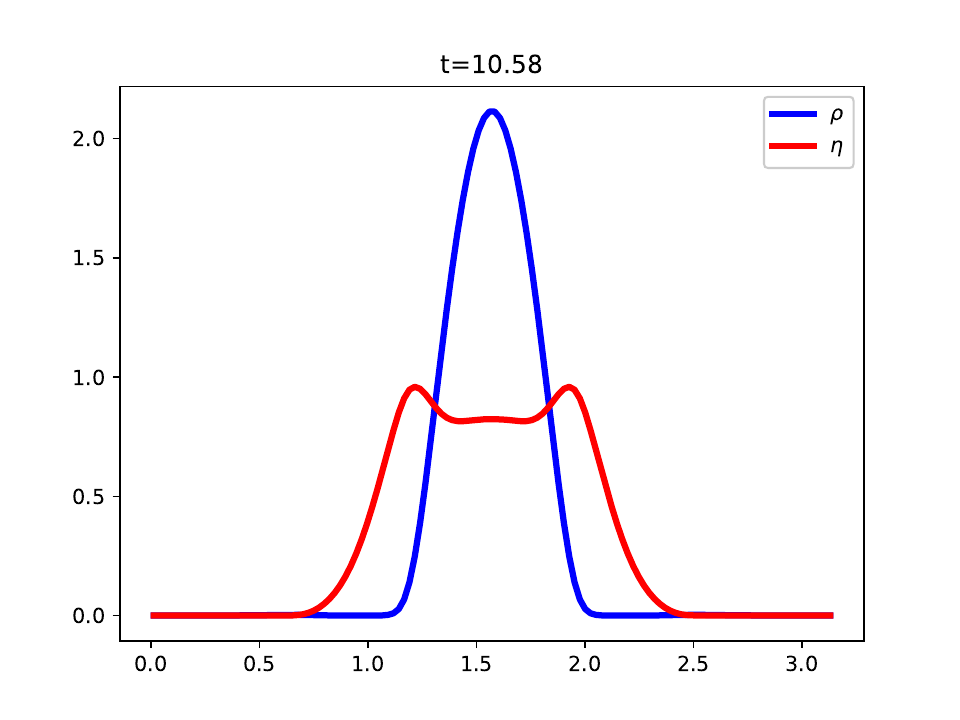}
	}
	\subfigure[$t=10.71$]{
	\includegraphics[width=0.31\textwidth]{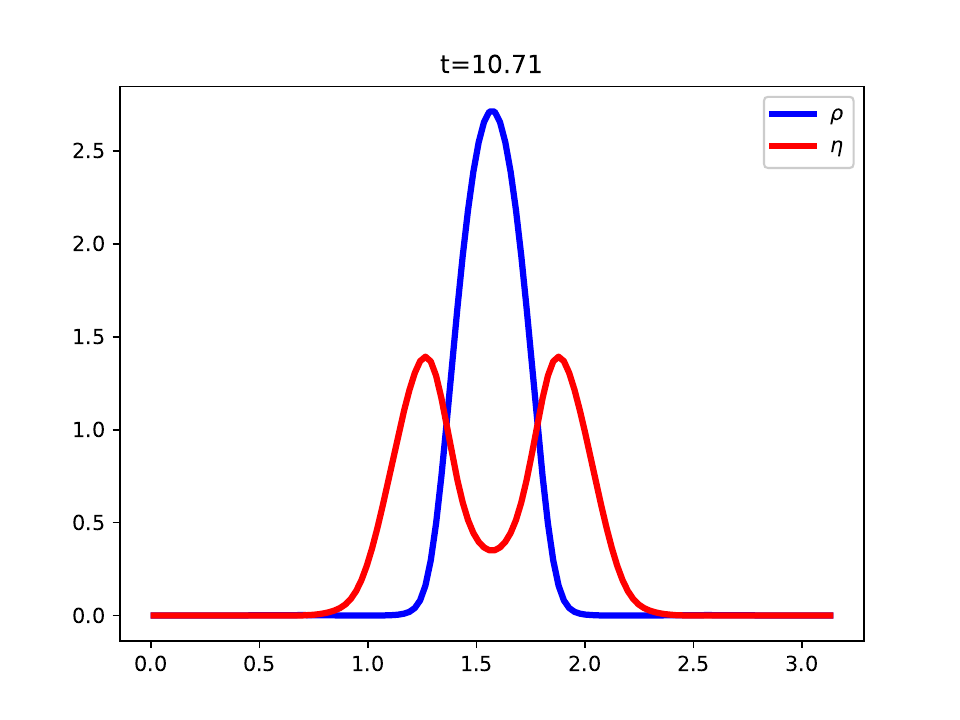}
	}
	\subfigure[$t = 19.54$]{
	\includegraphics[width=0.31\textwidth]{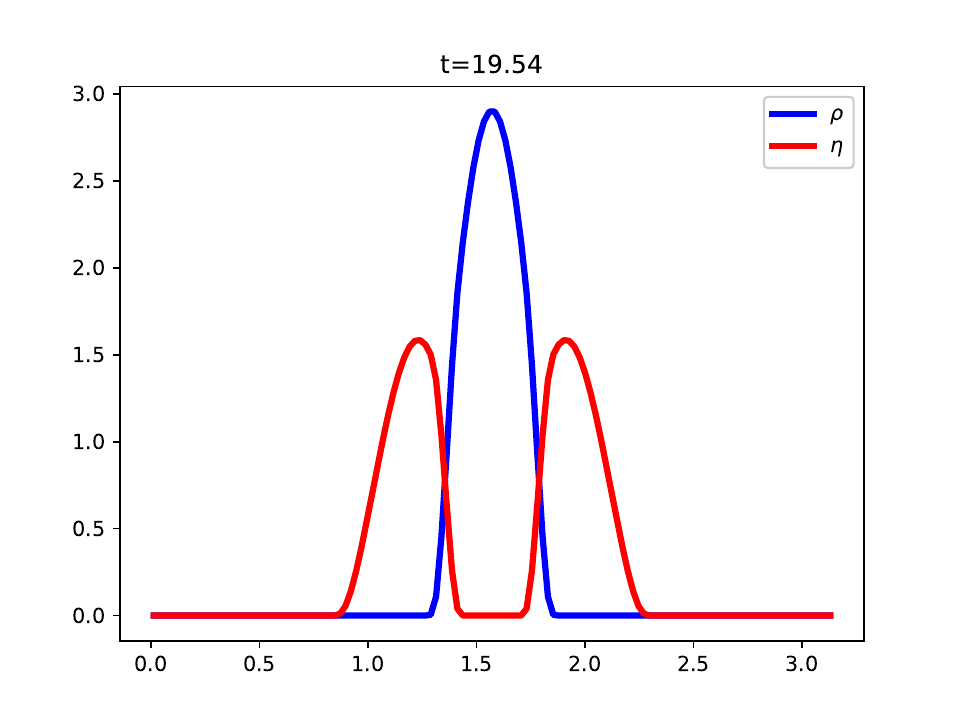}
	}
	\subfigure[Decay of the energy, $\frak{E}(\rho, \eta)$.]{
	\includegraphics[width=0.31\textwidth]{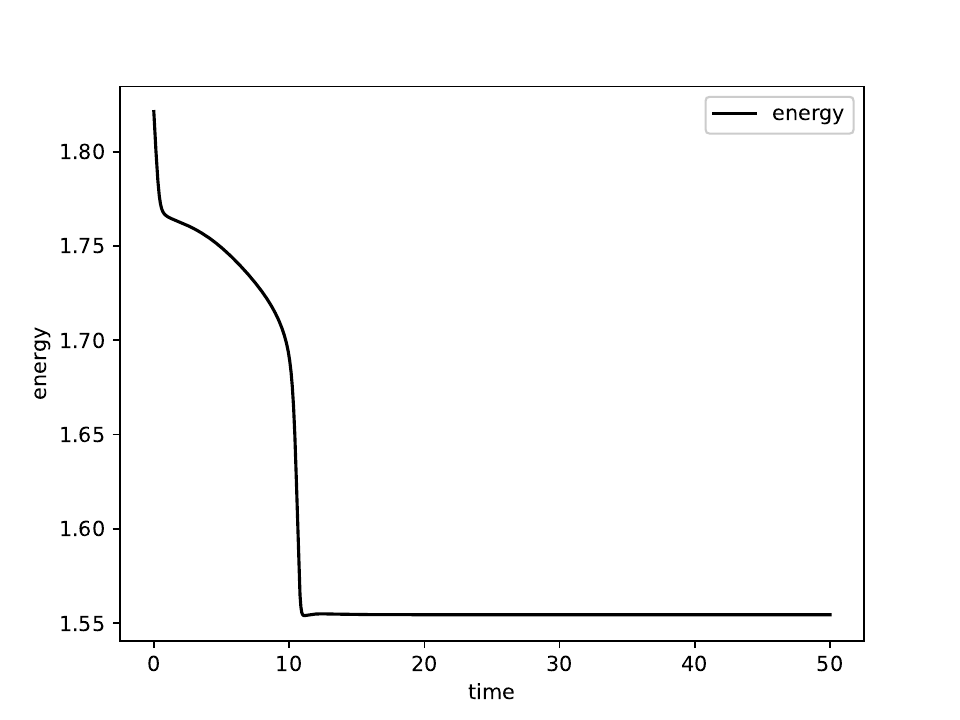}
	}
	\caption{Evolution of mixed initial data, $\rho = \sin(2x)^2$ and $\eta = \cos(2x)^2$ on the domain $(0, \pi)$. The interactions are linear combinations of Gaussians modelling self-attraction whereas the cross-interactions are short-range repulsive and long-range attractive. The associated energy decays abruptly at the merging and separation between aggregates. The slower decay of the energy before $t\approx 10$ is due to the trade off between the local (internal) energy and the nonlocal interaction energy. After the rearrangement of the five initial clusters to only three, the energy stabilises.}
	\label{fig:gaussianpotentials}
\end{figure}

\section{Conclusion}
In this paper we presented a finite volume scheme for a system of non-local partial differential equations with cross-diffusion. We were able to reproduce a continuous energy estimate on the discrete level for our scheme. These discrete estimates for the approximate solution are enough to get compactness results and we are able to identify the limit of the approximate solutions as a weak solution of the equation. We complement the analytical part with numerical simulations. These back up our convergence result and we are also able to apply the scheme in cases in which we cannot show convergence. To this end we pushed the scheme to regimes of singular potentials also lacking the regularising porous medium type self-diffusion terms. Comparing them with the explicit stationary states from \cite{CHS17} we conclude the scheme performs well even in  regimes it was not designed for.

\section*{Acknowledgements}
JAC was partially supported by EPSRC grant number EP/P031587/1.


\bibliographystyle{abbrv}
\bibliography{references}

\end{document}